\documentclass{article}
\usepackage[utf8]{inputenc}
\usepackage{amsmath}
\usepackage{amsthm}
\usepackage{amssymb}
\usepackage{amsfonts}
\usepackage{authblk}
\usepackage{bbm}
\usepackage[font=footnotesize]{caption}
\usepackage{enumitem}
\usepackage[scr=boondoxo]{mathalpha}
\usepackage{mathtools}
\usepackage{multicol}
\usepackage{stmaryrd}
\usepackage{tikz-cd}
\usepackage[
    colorlinks=true,
    citecolor=black,
    urlcolor=black,
    linkcolor=black,
    breaklinks=true
]{hyperref}
\usepackage[
    capitalize
]{cleveref}
\usepackage[
    backend=biber,
    style=alphabetic,
    sorting=nty,
    doi=false,
    url=false,
    isbn=false
]{biblatex}
\usepackage[
    a4paper,
    top = 1in,
    bottom = 1.25in,
    left = 1.25in,
    right = 1.25in
]{geometry}
\usepackage{xurl}

\usetikzlibrary{patterns}

\setlist[enumerate]{
    leftmargin=2em,
    label=(\roman*)
}
\setlist[itemize]{
    align=parleft,
    left=0pt..1em,
    itemsep=2pt,
    topsep=4pt
}

\newtheorem{theorem}{Theorem}[section]
\newtheorem*{theorem*}{Theorem}
\newtheorem{proposition}[theorem]{Proposition}
\newtheorem{lemma}[theorem]{Lemma}
\newtheorem{corollary}[theorem]{Corollary}
\newtheorem{theoremx}{Theorem}

\theoremstyle{definition}
\newtheorem{example}[theorem]{Example}
\newtheorem{remark}[theorem]{Remark}

\newtheorem{definition}[theorem]{Definition}

\newcommand{\ZZ}{\mathbb{Z}}
\newcommand{\QQ}{\mathbb{Q}}
\newcommand{\RR}{\mathbb{R}}
\newcommand{\CC}{\mathbb{C}}

\newcommand{\id}{\textup{id}}

\newcommand{\uQQ}{\underline{\QQ}}

\newcommand{\KK}{\mathbb{K}}
\newcommand{\uKK}{\underline{\KK}}

\newcommand{\cat}[1]{\textup{#1}}

\newcommand{\Mod}{\cat{Mod}}

\newcommand{\Var}{\cat{Var}}
\newcommand{\Stck}{\cat{Stck}}

\newcommand{\TopStck}{\cat{LCHStck}}
\newcommand{\Bord}{\cat{Bord}}
\newcommand{\LCH}{\cat{LCH}}
\newcommand{\Cat}{\cat{Cat}}

\newcommand{\K}{\textup{K}_0}

\newcommand{\GL}{\textup{GL}}
\newcommand{\SL}{\textup{SL}}
\newcommand{\SU}{\textup{SU}}
\newcommand{\PU}{\textup{PU}}
\newcommand{\U}{\textup{U}}
\newcommand{\SO}{\textup{SO}}
\newcommand{\PGL}{\textup{PGL}}

\newcommand{\defequality}{\coloneqq}

\DeclareMathOperator{\Hom}{Hom}
\newcommand{\iHom}{\mathscr{H}\!\textup{om}}
\DeclareMathOperator{\Ext}{Ext}

\DeclareMathOperator{\Corr}{Corr}

\DeclareMathOperator{\Sh}{Sh}
\DeclareMathOperator{\Cone}{Cone}

\newcommand{\Rep}[2]{R_{#1}({#2})}
\newcommand{\CharStck}[2]{\mathfrak{X}_{#1}({#2})}
\newcommand{\RepTw}[3]{R^\textup{tw}_{#1}(#2; #3)}
\newcommand{\CharStckTw}[3]{\mathfrak{X}^\textup{tw}_{#1}(#2; #3)}

\newcommand{\smatrix}[1]{\left(\begin{smallmatrix} #1 \end{smallmatrix}\right)}

\renewenvironment{cases}{%
\left\{\begin{array}{cl}%
}{\end{array}\right.}

\newcommand{\bdscale}{0.5}

\newcommand{\bdidentity}[1][1] {
\begin{tikzpicture}[semithick, scale=#1*\bdscale, baseline=-0.5ex]
\begin{scope}
    \draw (-1,0) ellipse (0.2cm and 0.4cm);
    \draw (-1,0.4) -- (1,0.4);
    \draw (-1,-0.4) -- (1,-0.4);
    \draw (1,0) ellipse (0.2cm and 0.4cm);
\end{scope}
\end{tikzpicture}
}

\newcommand{\bdunit}[1][1]{
\begin{tikzpicture}[semithick, scale=#1*\bdscale, baseline=-0.5ex]
\begin{scope}
    \draw (0,0) ellipse (0.2cm and 0.4cm);
    \draw (0,-0.4) arc (-90:90:0.75cm and 0.4cm);
\end{scope}
\end{tikzpicture}
}

\newcommand{\bdcounit}[1][1]{
\begin{tikzpicture}[semithick, scale=#1*\bdscale, baseline=-0.5ex]
\begin{scope}
    \draw (0,0) ellipse (0.2cm and 0.4cm);
    \draw (0,0.4) arc (90:270:0.75cm and 0.4cm);
\end{scope}
\end{tikzpicture}
}

\newcommand{\bdgenus}[1][1]{
\begin{tikzpicture}[semithick, scale=#1*\bdscale, baseline=-0.5ex]
\begin{scope}
    \draw (-1,0) ellipse (0.2cm and 0.4cm);
    \draw (-1,0.4) .. controls (-0.5,0.6) and (0.5,0.6) .. (1,0.4);
    \draw (-1,-0.4) .. controls (-0.5,-0.6) and (0.5,-0.6) .. (1,-0.4);
    \draw (-0.5,0.1) .. controls (-0.5,-0.125) and (0.5,-0.125) .. (0.5,0.1);
    \draw (-0.4,0.0) .. controls (-0.4,0.0625) and (0.4,0.0625) .. (0.4,0.0);
    \draw (1,0) ellipse (0.2cm and 0.4cm);
\end{scope}
\end{tikzpicture}
}

\newcommand{\bdcomultiplication}[1][1]{
\begin{tikzpicture}[semithick, scale=#1*\bdscale, baseline=-0.5ex]
\begin{scope}
    \useasboundingbox (-1.3,-.75) rectangle (1.2,.75);
    \draw (-1,0.5) ellipse (0.2cm and 0.4cm);
    \draw (-1,-0.5) ellipse (0.2cm and 0.4cm);
    \draw (1,0) ellipse (0.2cm and 0.4cm);
    \draw (-1,0.9) .. controls (0,0.9) and (0,0.4) .. (1,0.4);
    \draw (-1,-0.9) .. controls (0,-0.9) and (0,-0.4) .. (1,-0.4);
    \draw (-1,0.1) .. controls (-0.2,0.1) and (-0.2,-0.1) .. (-1,-0.1);
\end{scope}
\end{tikzpicture}
}

\newcommand{\bdmultiplication}[1][1]{
\begin{tikzpicture}[semithick, scale=#1*\bdscale, baseline=-0.5ex]
\begin{scope}
    \useasboundingbox (-1.2,-.75) rectangle (1.3,.75);
    \draw (-1,0) ellipse (0.2cm and 0.4cm);
    \draw (1,0.5) ellipse (0.2cm and 0.4cm);
    \draw (1,-0.5) ellipse (0.2cm and 0.4cm);
    \draw (-1,0.4) .. controls (0,0.4) and (0,0.9) .. (1,0.9);
    \draw (-1,-0.4) .. controls (0,-0.4) and (0,-0.9) .. (1,-0.9);
    \draw (1,0.1) .. controls (0.2,0.1) and (0.2,-0.1) .. (1,-0.1);
\end{scope}
\end{tikzpicture}
}

\newcommand{\bdtwist}[1][1] {
\begin{tikzpicture}[semithick, scale=#1*\bdscale, baseline=-0.5ex]
\begin{scope}
    \draw (-1,0.5) ellipse (0.2cm and 0.4cm);
    \draw (-1,-0.5) ellipse (0.2cm and 0.4cm);
    \draw (1,0.5) ellipse (0.2cm and 0.4cm);
    \draw (1,-0.5) ellipse (0.2cm and 0.4cm);
    \draw (-1,0.9) .. controls (0,0.9) and (0,-0.1) .. (1,-0.1);
    \draw (-1,0.1) .. controls (0,0.1) and (0,-0.9) .. (1,-0.9);
    \draw (-1,-0.9) .. controls (0,-0.9) and (0,0.1) .. (1,0.1);
    \draw (-1,-0.1) .. controls (0,-0.1) and (0,0.9) .. (1,0.9);
\end{scope}
\end{tikzpicture}
}

\def\bdgenerictube#1 {
\begin{tikzpicture}[semithick, scale=\bdscale, baseline=-0.5ex]
\begin{scope}
    \draw (-1,0) ellipse (0.2cm and 0.4cm);
    \draw (-1,0.4) cos (-0.875, 0.5) sin (-0.75,0.6) cos (-0.625,0.5) sin (-0.5,0.4) cos (-0.375,0.5) sin (-0.25, 0.6) cos (-0.125,0.5) sin (0, 0.4) cos (0.125,0.5) sin (0.25,0.6) cos (0.375,0.5) sin (0.5,0.4) cos (0.625,0.5) sin (0.75,0.6) cos (0.875,0.5) sin (1,0.4);
    \draw (-1,-0.4) cos (-0.875, -0.5) sin (-0.75,-0.6) cos (-0.625,-0.5) sin (-0.5,-0.4) cos (-0.375,-0.5) sin (-0.25, -0.6) cos (-0.125,-0.5) sin (0,-0.4) cos (0.125,-0.5) sin (0.25,-0.6) cos (0.375,-0.5) sin (0.5,-0.4) cos (0.625,-0.5) sin (0.75,-0.6) cos (0.875,-0.5) sin (1,-0.4);
    \draw (1,0) ellipse (0.2cm and 0.4cm);
    \draw (0,0) node {$#1$};
\end{scope}
\end{tikzpicture}
}

\newsavebox{\bdgenusbox}
\sbox{\bdgenusbox}{\!\!\bdgenus[0.7]\!\!}
\newsavebox{\bdunitbox}
\sbox{\bdunitbox}{\!\!\bdunit[0.7]\!\!}

\let\oldparagraph=\paragraph
\renewcommand\paragraph[1]{\oldparagraph{#1.}}

\title{\textbf{Cohomology of character stacks via TQFTs}}
\author{Jesse Vogel}
\affil{\footnotesize Mathematical Institute, Leiden University, j.t.vogel@math.leidenuniv.nl}
\date{}

\begin{document}

\maketitle
\begin{abstract}
    We study the cohomology of $G$-representation varieties and $G$-character stacks by means of a topological quantum field theory (TQFT). This TQFT is constructed as the composite of a so-called field theory and the $6$-functor formalism of sheaves on topological stacks. We apply this framework to compute the cohomology of various $G$-representation varieties and $G$-character stacks of closed surfaces for $G = \SU(2), \SO(3)$ and $\U(2)$. This work can be seen as a categorification of earlier work, in which such a TQFT was constructed on the level of Grothendieck groups to compute the corresponding Euler characteristics. 
\end{abstract}


\section{Introduction}
\label{sec:introduction}

Let $\Gamma = \langle x_1, \ldots, x_n \mid r_1, \ldots, r_m \rangle$ be a finitely presented group with generators $x_i$ and relations $r_i$. Given a group $G$, one can consider the set of representations of $\Gamma$ into $G$
\[ \Rep{G}{\Gamma} \defequality \Hom(\Gamma, G) . \]
Identifying a representation $\rho \in \Rep{G}{\Gamma}$ with the tuple $(\rho(x_1), \ldots, \rho(x_n)) \in G^n$, the set $\Rep{G}{\Gamma}$ corresponds to the subset of $G^n$ consisting of tuples $(g_1, \ldots, g_n)$ such that $r_i(g_1, \ldots, g_n) = 1$ for all $i = 1, \ldots, m$.
In particular, when $G$ is a topological group, this equips $\Rep{G}{\Gamma}$ with the subspace topology, and when $G$ is an algebraic group, this equips $\Rep{G}{\Gamma}$ with the structure of an algebraic variety. Generally, when $G$ is a group object in a category $\mathcal{C}$ that admits finite limits, one can construct $\Rep{G}{\Gamma}$ as the fiber product
\[ \begin{tikzcd}
    \Rep{G}{\Gamma} \arrow{r} \arrow{d} & G^{n} \arrow{d}{r} \\ 1 \arrow{r} & G^{m}
\end{tikzcd} \]
where $r = (r_1, \ldots, r_m)$. One can show that $\Rep{G}{\Gamma}$ is independent (up to isomorphism) of the presentation of $\Gamma$.
Note that $G$ naturally acts on $\Rep{G}{\Gamma}$ by conjugation
\[ G \times \Rep{G}{\Gamma} \to \Rep{G}{\Gamma}, \quad (g, \rho) \mapsto g \rho g^{-1} , \]
identifying isomorphic representations. The quotient of $\Rep{G}{\Gamma}$ by $G$ is usually called the \emph{$G$-character variety} of $\Gamma$, and when one enters the realm of stacks, the quotient stack
\[ \CharStck{G}{\Gamma} \defequality \Rep{G}{\Gamma} / G \]
is called the \emph{$G$-character stack}.

Typically, $\Gamma$ is the fundamental group $\pi_1(M, *)$ of a connected closed manifold $M$ (in fact, every finitely presented group arises in this way), and in this case we also write $\Rep{G}{M}$ instead of $\Rep{G}{\pi_1(M, *)}$. When $G$ is a topological group, $\Rep{G}{M}$ parametrizes $G$-local systems (or $G$-torsors) on $M$ \cite[Theorem 2.5.15]{Szamuely2009}, and $\CharStck{G}{M}$ parametrizes them up to isomorphism.

\paragraph{Non-abelian Hodge theory}

A particularly well-studied case is that of the fundamental group of a closed orientable surface $\Sigma_g$ of genus $g$,
\[ \Gamma = \pi_1(\Sigma_g, *) = \langle a_1, b_1, \ldots, a_g, b_g \mid [a_1, b_1] \cdots [a_g, b_g] = 1 \rangle \]
where $[a_i, b_i] = a_i b_i a_i^{-1} b_i^{-1}$ denotes the group commutator. The corresponding $G$-character variety plays an important role in non-abelian Hodge theory: if $\Sigma_g$ is the underlying space of a smooth projective curve $C$, and $G$ is a semisimple complex algebraic group, then the character variety is real analytically isomorphic to a certain moduli space of Higgs bundles of on $C$ and a moduli space of flat connections on $C$ \cite{Corlette1988, Donaldson1987, Simpson1991, Simpson1994a}. In particular, the cohomology of all three moduli spaces coincide.

Under these correspondences, Hitchin computed the Poincaré polynomial of the $G$-character variety of $\Sigma_g$ for $G = \SL_2(\CC)$ \cite{Hitchin1987}, Gothen for $G = \SL_3(\CC)$ \cite{Gothen1994}, and García-Prada, Heinloth and Schmitt for $G = \GL_4(\CC)$ \cite{GarciaPradaHeinlothSchmitt2014}.
In recent years, these moduli spaces have been the subject of extensive research, and many new methods were developed to compute their cohomology or cohomology-like invariants.



\paragraph{Topological quantum field theories}

The method that is of particular interest to us is the method initiated by González-Prieto, Logares and Muñoz \cite{GonzalezLogaresMunoz2020}. They constructed a \emph{topological quantum field theory (TQFT)} that quantizes the virtual Hodge--Deligne polynomials of the representation varieties $\Rep{G}{\Sigma_g}$ for a complex algebraic group $G$. Let us explain what this means. Originating from physics, a TQFT is a (lax) symmetric monoidal functor $Z \colon \Bord_n \to \Mod_R$ from the category $\Bord_n$ of $n$-dimensional bordisms to the category $\Mod_R$ of $R$-modules over some commutative ring $R$, where the monoidal structures are given by the disjoint union of manifolds and the tensor product $\otimes_R$, respectively. Given a closed $n$-dimensional manifold $M$, one can view $M$ as a bordism from and to the empty manifold $\varnothing$ to produce a morphism $Z(M) \colon Z(\varnothing) \to Z(\varnothing)$. Since $Z(\varnothing) = R$ by monoidality, the $R$-linear map $Z(M)$ is multiplication by the element $Z(M)(1) \in R$. In other words, a TQFT $Z$ associates to every closed manifold $M$ an $R$-valued invariant $Z(M)(1)$, and we say that $Z$ \emph{quantizes} this invariant.

In \cite{GonzalezLogaresMunoz2020}, 
the authors constructed, for every complex algebraic group $G$, a TQFT
\begin{equation}
    \label{eq:tqft}
    Z_G \colon \Bord_n \to \K(\cat{MHS})\textup{-}\Mod ,
\end{equation}
where $R = \K(\cat{MHS})$ denotes the Grothendieck ring of mixed Hodge structures, such that $Z_G$ quantizes the class of the mixed Hodge structure on the $G$-representation variety of $M$.
To be precise, the original construction \cite{GonzalezLogaresMunoz2020} works with \emph{pointed bordisms} (which are bordisms with a specified set of basepoints that are used to keep track of non-trivial loops that arise when bordisms are glued), but in a later reformulation of the method \cite{GonzalezHablicsekVogel2023}, the $G$-representation variety was replaced by the $G$-character stack, which allowed to get rid of the need for specified basepoints as the automorphism groups of the stacky points contain enough information to keep track of the arising non-trivial loops.

Concretely, in dimension $n = 2$, the objects of $\Bord_2$ are (disjoint unions of) circles, and the morphisms are $2$-dimensional compact manifolds with boundary, called bordisms, connecting the objects. Composition of bordisms is performed by gluing common boundaries, and hence, a closed oriented surface $\Sigma_g$ of genus $g$ can be written as the composite of bordisms as follows:
\begin{equation}
    \label{eq:decomposition_sigma_g}
    \Sigma_g = \bdcounit \circ \underbrace{\bdgenus \circ \cdots \circ \bdgenus}_{g \textup{ times}} \circ \; \bdunit
\end{equation}
This reduces the problem of computing the invariants $Z_G(\Sigma_g)(1)$ to computing the linear maps $Z_G\left(\bdcounit[0.8]\right)$, $Z_G\left(\bdgenus[0.7]\right)$ and $Z_G\left(\bdunit[0.8]\right)$.
This method has been used to compute invariants of the $G$-representation varieties and $G$-character stacks of $\Sigma_g$ for $G = \SL_2(\CC)$ \cite{Gonzalez2020} and for $G$ equal to the groups of upper triangular matrices \cite{HablicsekVogel2022}.

\paragraph{Categorification}

We address two disadvantages of the TQFT method. Firstly, as the computed invariant lives in the Grothendieck ring of mixed Hodge structures, we are effectively only computing the `Euler characteristic' of the mixed Hodge structure of the $G$-character stack. Only in very special cases, such as for smooth projective varieties, can one infer the mixed Hodge structure itself from the virtual Hodge--Deligne polynomial. This is not the case for character stacks.
Secondly, even though both the category of bordisms and the category of stacks naturally admit a higher categorical structure (at least a $2$-categorical structure), this higher structure is not reflected in the TQFT: $Z_G$ makes use only of the $1$-categorical truncation of these categories.

The goal of this paper is improve on both points by \textit{categorifying} the TQFT method. That is, we will replace the $\K$-groups by the derived category. This will allow us to not only compute the Euler characteristic of the character stacks, but rather their whole cohomology.
Of course, this upgrade does not come for free: whereas for the computations in the $\K$-groups we may split every distinguished triangle that we encounter, for the computations in the derived category we have to deal with the connecting homomorphisms.

\paragraph{Six-functor formalisms}

The TQFT \eqref{eq:tqft} is constructed as the composite of two functors: a symmetric monoidal functor $F_G \colon \Bord_n \to \Corr(\Stck)$, called the \emph{field theory}, and a lax symmetric monoidal functor $Q \colon \Corr(\Stck) \to \K(\cat{MHS})\textup{-}\Mod$, called the \emph{quantization functor}. The field theory $F_G$ assigns to an object of $\Bord_n$, which is a closed $(n - 1)$-dimensional manifold $M$, the corresponding $G$-character stack $\CharStck{G}{M}$, and assigns to a bordism $W \colon M_1 \to M_2$, which is a compact $n$-dimensional manifold whose boundaries are $M_1$ and $M_2$, the correspondence of stacks
\[ \begin{tikzcd}[row sep=0em] & \arrow{ld} \CharStck{G}{W} \arrow{rd} \\ \CharStck{G}{M_1} & & \CharStck{G}{M_2} \end{tikzcd} \]
where the maps are induced by the inclusions $M_i \to W$. The quantization functor $Q$ assigns to a stack $\mathfrak{X}$ the Grothendieck group $\K(\cat{MHM}_\mathfrak{X})$ of mixed Hodge modules over $\mathfrak{X}$, and to a correspondence $\mathfrak{X} \xleftarrow{f} \mathfrak{Z} \xrightarrow{g} \mathfrak{Y}$ of stacks the morphism
\[ g_! f^* \colon \K(\cat{MHM}_\mathfrak{X}) \to \K(\cat{MHM}_\mathfrak{Y}) \]
induced by the inverse image functor $f^* \colon \cat{MHM}_\mathfrak{X} \to \cat{MHM}_\mathfrak{Z}$ and the direct image with compact support functor $g_! \colon \cat{MHM}_\mathfrak{Z} \to \cat{MHM}_\mathfrak{Y}$. 

Now, the categories $\Bord_n$ and $\Corr(\Stck)$ are naturally admit a $2$-categorical structure, and the functor $F_G$ can be promoted to a $2$-functor. However, this is not the case for the functor $Q$, as $\K(\cat{MHS})\textup{-}\Mod$ is only a $1$-category. For this reason, it makes sense to replace the Grothendieck groups $Q(\mathfrak{X}) = \K(\cat{MHM}_\mathfrak{X})$ by the derived categories $D(\cat{MHM}_\mathfrak{X})$ to obtain a lax symmetric monoidal $2$-functor
\begin{equation}
    \label{eq:quantization_functor_2_categorical}
    Q \colon \Corr(\Stck) \to \Cat
\end{equation}
to the $2$-category of categories. 

The categories of mixed Hodge modules are not special in this regard: given any $6$-functor formalism, one can construct lax symmetric monoidal functor as in \eqref{eq:quantization_functor_2_categorical}. In fact, such a functor is precisely what defines a $3$-functor formalism, following \cite[Definition A.5.10]{Mann2022} and \cite[Definition 2.4]{Scholze2022} (the three functors being $f^*$, $f_!$ and $\otimes$; a $6$-functor formalism is a $3$-functor formalism for which the three functors have right adjoints). This leads to a significant generalization of the TQFT method: for any group object $G$ in a category $\mathcal{C}$ and a $6$-functor formalism on $\mathcal{C}$, one constructs a `TQFT' $Z_G$ quantizing the corresponding cohomology of the $G$-character stack.

\paragraph{Derived category of sheaves}

For simplicity and concreteness, we will focus on the derived category of constructible sheaves on topological spaces. Given a topological space $X$, denote by $D(X) \defequality D(\Sh(X))$ the derived category of sheaves of abelian groups on $X$. For any morphism $f \colon X \to Y$ of topological spaces, we obtain the functors (by convention, all functors are derived)
\[ \begin{gathered}
    f^*, f^! \colon D(Y) \to D(X) \quad \textup{ and } \quad f_*, f_! \colon D(X) \to D(Y)
\end{gathered} \]
satisfying various compatibilities, such as the projection formula and the proper base change
\[ f_! (-) \otimes (-) \cong f_! ((-) \otimes f^* (-)) \quad \textup{ and } \quad g^* f_! \cong f'_! g'^* . \]
Details, and a lot more work on derived categories of sheaves, can be found in \cite{KashiwaraSchapira1990}.

Together, these functors form a $6$-functor formalism on the category of locally compact Hausdorff spaces \cite[Theorem 7.4]{Scholze2022}, which can be extended to the category of stacks over such spaces \cite[Theorem 4.20]{Scholze2022}. Composing with the field theory $F_G \colon \Bord_n \to \Corr(\TopStck)$ yields a lax TQFT $Z_G$ that quantizes the cohomology with compact support of the $G$-character stack.

\begin{theoremx}
    For any locally compact topological group $G$, there exists a lax symmetric monoidal functor $Z_G \colon \Bord_n \to \Cat_\infty$ that quantizes the cohomology with compact support of the $G$-character stack.
\end{theoremx}

\paragraph{Applications}

An an application, we will use the TQFT of sheaves to compute the cohomology of the $G$-representation variety and $G$-character stacks of the closed orientable surfaces $\Sigma_g$ for the group $G$ equal to $\SU(2)$, $\SO(3)$ or $\U(2)$ and any genus $g \ge 0$, and some variations thereof. An overview of the precise cohomology groups that are computed can be found in \cref{tab:main_theorems}.

\begin{table}[h!t]
    \renewcommand{\arraystretch}{1.5}
    \centering
    \begin{tabular}{|c|c|}
         \hline
         \cref{thm:poincare_polynomial_SU2_representation_varieties} & $H^*_{(c)}(\Rep{\SU(2)}{\Sigma_g}; \QQ)$ for all $g \ge 0$ \\
         \hline
         \cref{cor:poincare_polynomial_twisted_SU2_representation_varieties} & $H^*_{(c)}(\RepTw{\SU(2)}{\Sigma_g}{C}; \QQ)$ for all $g \ge 0$ and $C \ne 1$ \\
         \hline
         \cref{prop:expression_Sn} & $H^*_{(c)}(\Rep{\SU(2)}{N_r}; \QQ)$ for all $r \ge 0$ \\
         \hline
         \cref{cor:poincare_polynomial_SO3_representation_varieties} & $H^*_{(c)}(\Rep{\SO(3)}{\Sigma_g}; \QQ)$ for all $g \ge 0$ \\
         \hline
         \cref{prop:Un_in_terms_of_PUn} & $H^*_{(c)}(\Rep{\U(2)}{\Sigma_g}; \QQ)$ for all $g \ge 0$ \\
         \hline
         \cref{thm:cohomology_SU2_character_stacks} & $H^*_c(\CharStck{\SU(2)}{\Sigma_g}; \QQ)$ and $H^*(\CharStck{\SU(2)}{\Sigma_g}; \QQ)$ for all $g \ge 0$ \\
         \hline
         \cref{thm:cohomology_twisted_SU2_character_stacks} & $H^*_{(c)}(\CharStckTw{\SU(2)}{\Sigma_g}{C}; \QQ)$ for all $g \ge 0$ and $C \ne 1$ \\
         \hline 
    \end{tabular}
    \caption{The computational results of this paper. The symbol $\Sigma_g$ denotes the closed orientable surface of genus $g$, and $N_r$ denotes the non-orientable surface of demigenus $r$. Furthermore, we write $H^*$ for cohomology and $H^*_c$ cohomology with compact support, and we write $H^*_{(c)}$ when they agree (e.g.\ when the space is compact).}
    \label{tab:main_theorems}
\end{table}

The computation of the cohomology of these representation varieties and character stacks are particularly interesting for the following reasons.
First, because of the Narasimhan--Seshadri theorem \cite{NarasimhanSeshadri1965}, which states that stable vector bundles (resp.\ with unit determinant) on $\Sigma_g$ are in correspondence with unitary representations (resp.\ with unit determinant) of the fundamental group $\Gamma = \pi_1(\Sigma_g, *)$.
Secondly, as described by \cite{FlorentinoLawton2024}, for a certain class of so-called \textit{flawed groups} $\Gamma$, whenever $G$ is a complex reductive group with maximal compact subgroup $K$, the $G$-character variety of $\Gamma$ deformation retracts onto the $K$-character variety of $\Gamma$, so in particular their cohomology coincides. While there are many flawed groups (e.g.\ every free, finite or nilpotent group is flawed), the surface groups $\Gamma = \pi_1(\Sigma_g, *)$ are \textit{flawless} (for $g \ge 2$) \cite[Example 2.7]{FlorentinoLawton2024}. Therefore, even though the groups $\SU(2)$, $\SO(3)$ and $\U(2)$ are the maximal compact subgroups of the complex groups $\SL_2(\CC), \PGL_2(\CC)$ and $\GL_2(\CC)$, respectively, the results of this paper could not have been obtained from the cohomology of the representation varieties or character stacks for these complex groups.
%

Finally, an interesting observation can be made about the computed invariants.
Namely, a common theme in the $\K$-theoretic setting is that the $\K$-invariants for the $G$-representation varieties of $\Sigma_g$ satisfy a recurrence relation for increasing values of $g$ due to the linear map $Z_G \left( \bdgenus[0.7] \right)$ restricting to a finitely generated submodule of $Z_G \left( S^1 \right)$ (cf.\ \cite{GonzalezLogaresMunoz2020, Gonzalez2020, HablicsekVogel2022}). However, this phenomenon does not occur for the computed Poincaré polynomials of the representation varieties, that is, they do not satisfy a recurrence relation.

\section{Derived categories of sheaves}

Let us recall the basic properties of the derived category of sheaves on topological spaces. Fix a commutative ring $\KK$. Given a topological space $X$, denote by $\Sh(X)$ the category of sheaves of $\KK$-modules on $X$, and by $D^b(X) \defequality D^b(\Sh(X))$ the corresponding bounded derived category. Given a $\KK$-module $M$, denote by $\underline{M} \in D^b(X)$ the object corresponding to the constant sheaf on $X$. 
When $X$ is a point, we have $\Sh(X) = \Mod_\KK$ and we denote the constant sheaf $\underline{M}$ also simply by $M$. 
From now on, we assume that all topological spaces are locally compact and Hausdorff.

Given a continuous map $f \colon X \to Y$, one defines the \emph{direct image functor}
\[ f_* \colon \Sh(X) \to \Sh(Y), \quad (f_* \mathcal{F})(V) = \mathcal{F}(f^{-1}(V)) , \]
and the \emph{inverse image functor} 
\[ f^* \colon \Sh(Y) \to \Sh(X) , \]
where, for a sheaf $\mathcal{G}$ on $Y$, the sheaf $f^* \mathcal{G}$ on $X$ is the sheaf associated to the presheaf $U \mapsto \varinjlim_{V \supseteq f(U)} \mathcal{G}(V)$ \cite[Definition 2.3.1]{KashiwaraSchapira1990}. These functors form an adjoint pair $f^* \dashv f_*$  \cite[Proposition 2.3.3]{KashiwaraSchapira1990}.
While $f^*$ is always exact \cite[Example 2.3.2]{KashiwaraSchapira1990}, the functor $f_*$ is only left exact, and exact in special cases such as when $f$ is a closed immersion. Hence, these functors induce derived functors between $D^b(X)$ and $D^b(Y)$, which we also denote by $f_*$ and $f^*$.

Furthermore, one defines the \emph{direct image functor with compact support} \cite[(2.5.1)]{KashiwaraSchapira1990}
\[ f_! \colon \Sh(X) \to \Sh(Y) , \quad (f_! \mathcal{F})(V) = \left\{ s \in \mathcal{F}(f^{-1}(V)) \mid f \colon \operatorname{supp}{(s)} \to V \textup{ is proper} \right\} , \]
which is left exact, and induces a right derived functor $f_! \colon D^b(X) \to D^b(Y)$. In case $f$ is the immersion of a locally closed subset, then $f_!$ is exact. When $f$ is proper, one has $f_! = f_*$.

When $f_!$ has finite cohomological dimension \cite[(3.1.3)]{KashiwaraSchapira1990}, the functor $f_!$ admits a right adjoint $f^! \colon D^b(Y) \to D^b(X)$ called the \emph{exceptional inverse image functor} (which is in general not the derived functor of a functor of sheaves). When $f$ is an open immersion, one has $f^! = f^*$.

When $f \colon X \to *$ is the projection to a point, the underived functors $f_*$ and $f_!$ are given by taking global sections (resp.\ with compact support). Hence, the derived functors $f_*$ and $f_!$ correspond to sheaf cohomology (resp.\ with compact support), that is,
\[ f_* \mathcal{F} = H^*(X, \mathcal{F}) \quad \textup{ and } \quad f_! \mathcal{F} = H_c^*(X, \mathcal{F}) . \]

The assignments from $f$ to $f^*, f_*, f_!$ and $f^!$ are all functorial, meaning that, given continuous maps $f \colon X \to Y$ and $g \colon Y \to Z$, we have natural isomorphisms \cite[(2.3.9), (2.6.5), (2.6.6), Proposition 3.1.8]{KashiwaraSchapira1990}
\[ (g f)^* \cong f^* g^*, \quad (g f)_* \cong g_* f_*, \quad (g f)_! \cong g_! f_! \quad \textup{and} \quad (g f)^! \cong f^! g^! . \]

The tensor product $\mathcal{F} \otimes \mathcal{G}$ of two sheaves $\mathcal{F}$ and $\mathcal{G}$ over $X$ is the sheaf associated to the presheaf $U \mapsto \mathcal{F}(U) \otimes_\KK \mathcal{G}(U)$ \cite[Definition 2.2.8]{KashiwaraSchapira1990}. The tensor product is right exact in both arguments, and induces a left derived tensor product on $D^b(X)$, also denoted by $\otimes$, when $\KK$ has \textit{finite weak global dimension} (meaning there exists an integer $n$ such that every $\KK$-module has a flat resolution of length $\le n$) \cite[p.110]{KashiwaraSchapira1990}. We will assume $\KK$ has this property.

Given sheaves $\mathcal{F}$ and $\mathcal{G}$ over $X$, one defines $\iHom(\mathcal{F}, \mathcal{G})$ as the sheaf over $X$ given by $U \mapsto \Hom_{\Sh(U)}(\mathcal{F}|_U, \mathcal{G}|_U)$ \cite[Definition 2.2.7]{KashiwaraSchapira1990}. The functor $\iHom$ is left exact in its second argument and induces a right derived functor $\iHom \colon D^-(X)^\textup{op} \times D^b(X) \to D^b(X)$. The functors $\iHom$ and $\otimes$ satisfy the usual tensor-hom adjunction.

Given spaces $X$ and $Y$, the \emph{external tensor product} is given by
\begin{equation}
    \label{eq:external_tensor_product}
    \boxtimes \colon D^b(X) \times D^b(Y) \to D^b(X \times Y), \quad (\mathcal{F}, \mathcal{G}) \mapsto \pi_X^* \mathcal{F} \otimes \pi_Y^* \mathcal{G} ,
\end{equation}
where $\pi_X \colon X \times Y \to X$ and $\pi_Y \colon X \times Y \to Y$ are the projections. Since the functors $\pi_X^*, \pi_Y^*$ and $\otimes$ are all exact, so is $\boxtimes$.

The following propositions provide a number of compatibilities between the functors described above that will be used throughout this paper. Note, however, that this is by far not an exhaustive list of compatibilities between the functors.

\begin{proposition}[{Proper base change \cite[Proposition 2.6.7]{KashiwaraSchapira1990}}]
    \label{prop:proper_base_change}
    Given a cartesian square 
    \begin{equation}
        \label{eq:proper_base_change_diagram}
        \begin{tikzcd}
            X \times_Z Y \arrow{r}{f'} \arrow{d}{g'} & Y \arrow{d}{g} \\ 
            X \arrow{r}{f} & Z
        \end{tikzcd}
    \end{equation}
    there is a canonical natural isomorphism of functors $f^* g_! \cong (g')_! (f')^*$.
\end{proposition}

\begin{proposition}[{Projection formula \cite[Proposition 2.6.6]{KashiwaraSchapira1990}}]
    Given a continuous map $f \colon X \to Y$ and objects $\mathcal{F} \in D^b(X)$ and $\mathcal{G} \in D^b(Y)$, there is a isomorphism
    \[ f_! \mathcal{F} \otimes \mathcal{G} \cong f_! \left(\mathcal{F} \otimes f^* \mathcal{G} \right) \]
    which is natural in $\mathcal{F}$ and $\mathcal{G}$.
\end{proposition}

\begin{proposition}[{Localization triangle \cite[(2.6.33)]{KashiwaraSchapira1990}}]
    \label{prop:localization_distinguished_triangle}
    Let $i \colon Z \to X$ be a closed immersion with open complement $j \colon U \to X$. Then for every object $\mathcal{F} \in D^b(X)$, there is a distinguished triangle
    \[ j_! j^* \mathcal{F} \to \mathcal{F} \to i_* i^* \mathcal{F} \xrightarrow{+} \]
    in $D^b(X)$, where the first two maps are given by the counit of the adjunction $j_! \dashv j^*$ and the unit of the adjunction $i^* \dashv i_*$, respectively.
\end{proposition}

\begin{remark}
    Since the category $\Sh(X)$ of sheaves of $\KK$-modules on a topological space $X$ is a Grothendieck abelian category, the six derived functors $f^*, f_*, f_!, f^!, \otimes$ and $\iHom$ are also well-defined on the unbounded derived category $D(X) \coloneqq D(\Sh(X))$.
\end{remark}

\subsection{Six-functor formalisms}

We wish to encode the above functors in the framework of a $6$-functor formalism. To do this, we follow the approach of \cite[Appendix A.5]{Mann2022} and \cite{Scholze2022} for the definitions of a $3$- and $6$-functor formalism. 
Consider the following definitions.

\begin{definition}
    A \emph{geometric setup} is a pair $(\mathcal{C}, E)$ consisting of an ($\infty$-)category $\mathcal{C}$ with finite limits and a collection of morphisms $E$ stable under pullback and composition containing all isomorphisms.
\end{definition}

\begin{definition}
    Let $(\mathcal{C}, E)$ be a geometric setup. Denote by $\Corr(\mathcal{C}, E)$ the \emph{$\infty$-category of correspondences} as defined in \cite[Definition 3.2]{Scholze2022}. In particular, the objects of $\Corr(\mathcal{C}, E)$ are the objects of $\mathcal{C}$, and the $1$-cells from $X$ to $Y$ are given by \emph{correspondences}, that is, diagrams $X \xleftarrow{f} Z \xrightarrow{g} Y$ in $\mathcal{C}$ with $g \in E$. The $2$-cells are given by diagrams
    \[ \begin{tikzcd}[row sep=0em] && W'' \arrow{dl} \arrow{dr} && \\ & W \arrow{dl} \arrow{dr} && W' \arrow{dl} \arrow{dr} & \\ X && Y && Z \end{tikzcd} \]
    with all arrows going down-right being in $E$ and the middle square being cartesian. In particular, we can think of $X \leftarrow W'' \rightarrow Z$ as the composite of $X \leftarrow W \rightarrow Y$ and $Y \leftarrow W'' \rightarrow Z$. The $\infty$-category $\Corr(\mathcal{C}, E)$ admits a symmetric monoidal structure given by \cite[Definition 3.11]{Scholze2022} induced by the cartesian symmetric monoidal structure on $\mathcal{C}$.
\end{definition}

\begin{definition}[{\cite[Definition 3.1]{Scholze2022}}]
    Let $(\mathcal{C}, E)$ be a geometric setup. A \emph{$3$-functor formalism} on $(\mathcal{C}, E)$ is a lax symmetric monoidal functor
    \[ D \colon \Corr(\mathcal{C}, E) \to \Cat_\infty . \]
\end{definition}

    Note that any morphism $f \colon X \to Y$ in $\mathcal{C}$ induces correspondences $Y \xleftarrow{f} X \xrightarrow{\id_X} X$ and $X \xleftarrow{\id_X} X \xrightarrow{f} Y$ (provided $f \in E$) between $X$ and $Y$. Given a $3$-functor formalism $D$, we write
    \[ f^* \defequality D\big(Y \xleftarrow{f} X \xrightarrow{\id_X} X\big) \quad \textup{ and } \quad f_! \defequality D\big(X \xleftarrow{\id_X} X \xrightarrow{f} Y\big) . \]
    Any general correspondence $X \xleftarrow{f} Z \xrightarrow{g} Y$ is isomorphic to the composite of $(Z \xleftarrow{\id_Z} Z \xrightarrow{g} Y)$ and $(X \xleftarrow{f} Z \xrightarrow{\id_Z} Z)$, so that $D(X \xleftarrow{f} Z \xrightarrow{g} Y) \cong g_! f^*$. For any space $X$, the lax monoidal structure of $D$ and pullback along the diagonal of $X$ define a tensor product on $D(X)$,
    \[ \otimes \colon D(X) \times D(X) \to D(X \times X) \xrightarrow{\Delta_X^*} D(X) \]
    equipping the categories $D(X)$ with a symmetric monoidal structure (the unit object of $D(X)$ is obtained via the composite of the natural functor $1 \to D(1)$ and pullback $D(1) \xrightarrow{\pi^*} D(X)$ along the final morphism $X \xrightarrow{\pi} 1$). Now, it follows formally that the natural transformation $D(X) \times D(Y) \to D(X \times Y)$ for any $X$ and $Y$ is given by the \emph{external tensor product}
    \[ \boxtimes \colon D(X) \times D(Y) \to D(X \times Y), \quad A \boxtimes B = \pi_X^* A \otimes \pi_Y^* B . \]

\begin{definition}[{\cite[Definition A.5.7]{Mann2022}}]
    A \emph{$6$-functor formalism} is a $3$-functor formalism $D \colon \Corr(\mathcal{C}, E) \to \Cat_\infty$ such that each symmetric monoidal $\infty$-category $D(X)$ is closed and all the functors $f^* \colon D(Y) \to D(X)$ and $f_! \colon D(X) \to D(Y)$ admit right adjoints, which are denoted by $f_*$ and $f^!$, respectively.
\end{definition}

\begin{example}
    All compatibilities between the six functors turn out to be encoded in the above definitions. For instance, the proper base change theorem, \cref{prop:proper_base_change} is shown as follows: for any diagram as in \eqref{eq:proper_base_change_diagram}, the statement follows by applying $D$ to the isomorphism of correspondences
    \[ \big( X \xleftarrow{g'} X \times_Z Y \xrightarrow{f'} Y \big) \cong \big( Z \xleftarrow{g} Y \xrightarrow{\id_Y} Y \big) \circ \big( X \xleftarrow{\id_X} X \xrightarrow{f} Z \big) . \]
    See \cite[Proposition A.5.8]{Mann2022} for more compatibilities.
\end{example}

The category $\mathcal{C} = \LCH$ of locally compact Hausdorff spaces, together with the class $E$ of all morphisms, forms a geometric setup. From now on, the notation $D(X)$ refers to the $\infty$-categorical analogue of the derived category of sheaves on $X$ \cite[Lecture VII]{Scholze2022}.

\begin{theorem}[{\cite[Theorem 7.4]{Scholze2022}}]
    The assignment $X \mapsto D(X)$ for locally compact Hausdorff spaces extends to a $6$-functor formalism
    \[ D \colon \Corr(\LCH, E) \to \Cat_\infty \]
    where $E$ contains all morphisms.
\end{theorem}

For the purposes of this paper, this $6$-functor formalism must be extended to topological stacks. Denote by $\TopStck$ the category of stacks of locally compact Hausdorff spaces, with respect to the topology given by open subsets. Following \cite[Appendix to Lecture IV]{Scholze2022}, we find such an extension is possible.

\begin{theorem}[{\cite[Theorem 4.20]{Scholze2022}}]
    \label{prop:D_six_functor_formalism}
    The assignments $X \mapsto D(X)$ and $(X \xleftarrow{f} Z \xrightarrow{g} Y) \mapsto (g_! f^* \colon D(X) \to D(Y))$ for locally compact Hausdorff spaces extend to a $6$-functor formalism
    \[ D \colon \Corr(\TopStck, \tilde{E}) \to \Cat_\infty . \]
\end{theorem}

\begin{example}[{\cite{MO471768}}]
    \label{ex:group_homology_and_cohomology_SU2}
    Consider the quotient stack $BG = * / G$ for $G = \SU(2)$ acting trivially on a point, and let $q \colon * \to BG$ be the quotient map and $\pi \colon BG \to *$ the terminal map. 
    The direct image $\pi_* \uQQ = H^*(BG; \QQ)$ corresponds to the group cohomology of $G$ and can be computed as $H^*(\mathbb{HP}^\infty; \QQ) = \bigoplus_{n \ge 0} \QQ[-4]$, where $\mathbb{HP}^\infty$ is the classifying space of $G = \SU(2)$. In contrast, $\pi_! \uQQ$ can be computed as follows. Since $q$ is cohomologically smooth (it pulls back to the smooth map $G \to *$), and the composite $\pi q = \id$ is smooth, also $\pi$ is cohomologically smooth. Hence, $\pi^!$ equals $\pi^*$ up to a twist. To find this twist, note that $q^! \pi^! \QQ = (\pi q)^! = \QQ$ and also $q^! \pi^! \QQ = q^* \pi^! \QQ \otimes q^! \QQ$. Now, $q^!$ shifts by $\dim G$ to the left, so $f^! \QQ$ sit in cohomological degree $\dim G$. In particular, in that degree sits a $1$-dimensional representation of $G$ on a $\QQ$-vector space, but as $G$ is connected, this representation must be trivial, and hence $f^! \QQ = \QQ[-\dim G]$. Finally, $\pi_!$, being left adjoint to $\pi^! = \pi^*[- \dim G]$, is therefore given by group homology shifted by $\dim G = 3$ to the left. In particular, $\pi_! \uQQ = \bigoplus_{n \ge 0} \QQ[4n + 3]$.

    Note that unbounded complexes, both to the left and to the right, are a necessary consequence when passing to stacks.
\end{example}

\section{TQFT of sheaves}

In this section, we define the TQFT of sheaves as a lax monoidal functor from the $\infty$-category of bordisms to the $\infty$-category of $\infty$-categories. This functor will be constructed as the composite of two functors
\[ \begin{tikzcd} \Bord_n \arrow{r}{F_G} & \Corr(\TopStck, \tilde{E}) \arrow{r}{D} & \Cat_\infty \end{tikzcd} \]
the latter of which is the functor $D$ of \cref{prop:D_six_functor_formalism}. Let us first define the category of bordisms.

\begin{definition}
    Let $n \ge 1$. The \emph{bicategory of $n$-dimensional bordisms}, denoted $\Bord_n$, is given as follows.
    \begin{itemize}
        \item The objects are closed oriented $(n - 1)$-dimensional manifolds.
        \item Given two objects $M_1$ and $M_2$, a $1$-morphism from $M_1$ to $M_2$ is a \emph{bordism} from $M_1$ to $M_2$, that is, a compact oriented $n$-dimensional manifold with boundary $W$ together with an orientation-preserving diffeomorphism $\partial W \cong \overline{M_1} \sqcup M_2$, where $\overline{M_1}$ denotes the manifold $M_1$ with opposite orientation. 
        \item The $2$-morphisms between bordisms are diffeomorphisms that respect the boundaries.
        \item Given two bordisms $W \colon M_1 \to M_2$ and $W' \colon M_2 \to M_3$, the composite $W' \circ W$ is given by the gluing $W'' = W' \coprod_{M_2} W$, which we equip with a smooth structure such that the inclusions $W \to W''$ and $W' \to W''$ are diffeomorphisms onto their images. By \cite[Theorem 1.4]{Milnor1965}, such a smooth structure exists up to (non-unique) diffeomorphism.
        \item For any object $M$, the identity morphism $\id_M$ is given by the cylinder $W = M \times [0, 1]$.
    \end{itemize}
    The bicategory $\Bord_n$ can naturally be regarded an $\infty$-category via the \textit{Duskin nerve} \cite{Duskin2001}, and we denote this $\infty$-category by $\Bord_n$ as well. Furthermore, $\Bord_n$ admits a symmetric monoidal structure given by taking disjoint unions of manifolds.
\end{definition}

Since the objects of $\Bord_n$ are allowed to be non-connected, we must generalize the definitions of the $G$-representation variety and the $G$-character stack to such manifolds.

\begin{definition}
    Let $M$ be a connected compact manifold and $G$ be a locally compact topological group. Denote by $\Rep{G}{M}$ the \emph{$G$-representation variety} of $M$, as described in \cref{sec:introduction}, and by $\CharStck{G}{M}$ the \emph{$G$-character stack} of $M$, that is, the quotient stack $\Rep{G}{M} / G$ where $G$ acts on $\Rep{G}{M}$ by conjugation. For a general (not necessarily connected) compact manifold $M$, let
    \[ \Rep{G}{M} = \prod_{i \in \pi_0(M)} \Rep{G}{M_i} \quad \textup{ and } \quad \CharStck{G}{M} = \prod_{i \in \pi_0(M)} \CharStck{G}{M_i} \]
    where $M_i$ are the (finitely many) connected components of $M$. Note that $\CharStck{G}{M}$ is still the moduli stack of $G$-local systems on $M$, because to give a $G$-local system on $M$ is to give a $G$-local system on every $M_i$.
    Furthermore, note that the constructions of the $G$-representation variety and $G$-character stack are functorial: a smooth map of manifolds $M \to N$ induces morphisms $\Rep{G}{N} \to \Rep{G}{M}$ and $\CharStck{G}{N} \to \CharStck{G}{M}$, respecting composition.
\end{definition}

\begin{definition}
    Let $n \ge 1$. The \emph{field theory} for a locally compact topological group $G$ is the symmetric monoidal functor
    \[ F_G \colon \Bord_n \to \Corr(\TopStck, \tilde{E}) \]
    which sends an object $M$ to $\CharStck{G}{M}$, a bordism $W \colon M_1 \to M_2$ to the induced correspondence $\CharStck{G}{M_1} \leftarrow \CharStck{G}{W} \rightarrow \CharStck{G}{M_2}$, and a diffeomorphism $W \xrightarrow{\sim} W'$ between bordisms $W, W' \colon M_1 \to M_2$ to the induced $2$-cell
    \[ \begin{tikzcd}[row sep=0em] && \CharStck{G}{W \sqcup_{M_2} W'} \arrow{ld} \arrow{rd} && \\ & \CharStck{G}{W} \arrow{ld} \arrow{rd} && \CharStck{G}{W'} \arrow{ld} \arrow{rd} \\ \CharStck{G}{M_1} && \CharStck{G}{M_2} && \CharStck{G}{M_3} . \end{tikzcd} \]
    The middle square is indeed cartesian, since a $G$-local system on $W \sqcup_{M_2} W'$ is equivalent to a $G$-local system on $W$ and $W'$ together with an isomorphism between their restrictions to $M_2$. Furthermore, $F_G$ is indeed monoidal as $\mathfrak{X}_G$ sends disjoint unions to products.
\end{definition}

\begin{definition}
    Let $n \ge 1$. For any locally compact topological group $G$, define $Z_G \colon \Bord_n \to \Cat_\infty$ as the composite $D \circ F_G$ with $D$ as in \cref{prop:D_six_functor_formalism}.
\end{definition}

\begin{remark}
    Note that, for any bordism $W \colon M_1 \to M_2$, the functor $Z_G(W) \colon D(\CharStck{G}{M_1}) \to D(\CharStck{G}{M_2})$ produced by the TQFT is exact. 
\end{remark}

\begin{theorem}
    For any locally compact topological group $G$, the lax TQFT $Z_G$ quantizes the cohomology with compact support of the $G$-character stack. That is, for any closed $n$-dimensional manifold $W$, viewed as a bordism $\varnothing \to \varnothing$, the induced functor $Z_G(W) \colon D(*) \to D(*)$ is given by $(-) \otimes H_c^*(\CharStck{G}{W}; \KK)$.
\end{theorem}
\begin{proof}
    The correspondence $F_G(W)$ is given by
    \[ * = \CharStck{G}{\varnothing} \xleftarrow{\pi} \CharStck{G}{W} \xrightarrow{\pi} \CharStck{G}{\varnothing} = * \]
    and applying $D$ yields $Z_G(W) = \pi_! \pi^*$, which is precisely $(-) \otimes H_c^*(\CharStck{G}{W}; \KK)$.
\end{proof}


\subsection{TQFT for surfaces}

For the remainder of this paper, we will restrict to the case of surfaces, that is, dimension $n = 2$. By \cite[Proposition 1.4.13]{Kock2004}, the category $\Bord_2$ of $2$-dimensional bordisms is generated by the bordisms
\begin{equation*}
    \begin{gathered}
        \bdidentity \colon S^1 \to S^1, \quad \bdmultiplication[0.75] \colon S^1 \sqcup S^1 \to S^1, \quad \bdcomultiplication[0.75] \colon S^1 \to S^1 \sqcup S^1, \\[5pt] \bdcounit \colon S^1 \to \varnothing, \quad \bdunit \colon \varnothing \to S^1 \quad\!\! \textup{ and } \quad \bdtwist[0.75] \colon S^1 \sqcup S^1 \to S^1 \sqcup S^1
    \end{gathered}
\end{equation*}
in the sense that every morphism in $\Bord_2$ is ($2$-isomorphic to) the composite of disjoint unions of these bordisms.
In particular, one can understand the functor $Z_G$ by the images of these six generators.

\begin{proposition}
    \label{prop:field_theory_surfaces}
    The field theories for the generators of $\Bord_2$ are given by
    \begin{enumerate}[label=(\roman*)]
        \item $F_G \left( \bdidentity[0.6] \right) = \Big( G/G \xleftarrow{\id} G/G \xrightarrow{\id} G/G \Big)$,
        \item $F_G\left(\bdmultiplication[0.5]\right) = \Big( G/G \times G/G \xleftarrow{(\pi_1/G, \pi_2/G)} G^2/G \xrightarrow{m/G} G/G \Big)$ where $\pi_1, \pi_2 \colon G^2 \to G$ are the projections and $m \colon G^2 \to G$ is the multiplication map,
        \item $F_G\left(\bdcomultiplication[0.5]\right) = \Big( G/G \xleftarrow{m/G} G^2/G \xrightarrow{(\pi_1/G, \pi_2/G)} G/G \times G/G \Big)$,
        \item $F_G\left(\bdcounit[0.75]\right) = \Big(G/G \xleftarrow{i/G} BG \rightarrow * \Big)$ where $i \colon \{ 1 \} \to G$ the inclusion of the unit of $G$,
        \item $F_G \left(\bdunit[0.75] \right) = \Big(* \leftarrow BG \xrightarrow{i/G} G/G \Big)$,
    \end{enumerate}
    where $G$ acts by conjugation on itself and on $G^2$.
\end{proposition}
\begin{proof}
    From the fundamental groups $\pi_1(S^1, *) = \ZZ$, $\pi_1(D^1, *) = 1$ and $\pi_1(S^2 \setminus \{ \textup{3 pts} \} = F_2$ (the free group on two generators), follows that
    \[ \CharStck{G}{S^1} = G/G, \quad \CharStck{G}{D^1} = BG \quad \textup{ and } \quad \CharStck{G}{S^2 \setminus \{ \textup{3 pts} \}} = G^2/G . \]
    The morphisms between these character stacks, induced by the inclusions of the boundaries of the bordisms, can be obtained completely analogous to \cite[Propositions 4.8.2 and 4.8.3]{Vogel2024}.
\end{proof}

To compute the cohomology of the $G$-character stacks of the closed orientable surfaces $\Sigma_g$ of genera $g$, we may express $\Sigma_g$ in terms of the above generators, and apply $Z_G$ functorially. For instance, from the expression $\Sigma_g = \bdcounit[0.75] \circ \left(\bdmultiplication[0.5] \circ \bdcomultiplication[0.5]\right)^g \circ \bdunit[0.75]$ we obtain
\[ Z_G \left( \Sigma_g \right) = Z_G \left( \bdcounit \right) \circ \left( Z_G \left( \bdmultiplication[0.75] \right) \circ Z_G \left( \bdcomultiplication[0.75] \right) \right)^{g} \circ Z_G \left( \bdunit \right) \]
for all $g \ge 0$. However, the following alternative expression for $\Sigma_g$ turns out to be more useful:
\[ \begin{tikzpicture}
    \node at (-0.5, 0) {$\Sigma_g$};
    \node at (0.125, 0) {$=$};
    \node at (0.825, 0) {$\bdcounit$};
    \node at (1.215, 0) {$\circ$};
    
    \node at (2.0, 0) {$\bdmultiplication$};
    \node at (2.715, 0.25) {$\circ$};
    \node at (2.715, -0.25) {$\circ$};
    \node at (3.0, 0.3) {$W$};
    
    \node at (3.5, -0.25) {$\bdmultiplication$};
    \node at (4.215, 0.0) {$\circ$};
    \node at (4.215, -0.5) {$\circ$};
    \node at (4.5, 0.05) {$W$};
    
    \node at (5.0, -0.50) {$\bdmultiplication$};
    \node at (5.715, -0.25) {$\circ$};
    \node at (5.715, -0.75) {$\circ$};
    \node at (6.0, -0.20) {$W$};

    \node at (6.215, -0.825) {\rotatebox{-9.46}{$\cdots$}};

    \node at (6.715, -0.9) {$\circ$};
    \node at (7.5, -0.9) {$\bdmultiplication$};
    \node at (8.215, -0.65) {$\circ$};
    \node at (8.215, -1.15) {$\circ$};
    \node at (8.5, -0.60) {$W$};
    \node at (8.5, -1.10) {$W$};
\end{tikzpicture} \]
with $W = \bdmultiplication[0.5] \circ \bdcomultiplication[0.5] \circ \bdunit[0.75]$ appearing $g$ times. In particular, denoting the `multiplication map' $Z_G \left( \bdmultiplication[0.5] \right) \circ \boxtimes \colon D(G/G) \times D(G/G) \to D(G/G)$ by $*$ and the object $Z_G \left( W \right) (\uKK) \in D(G/G)$ by $\mathcal{F}$, we have
\begin{equation}
    \label{eq:ZG_Sigma_g_from_convolution_mathcal_F}
    Z_G \left( \Sigma_g \right) (\uKK) = Z_G \left( \bdcounit \right) \Big( \underbrace{\; \mathcal{F} * \cdots * \mathcal{F}}_{g \textup{ times}} \Big) .
\end{equation}

\subsection{Representation varieties}

Even though the TQFT $Z_G$ computes the cohomology of the $G$-character stack, it is possible to adapt the functors produced by $Z_G$ to obtain functors that compute the cohomology of the $G$-representation variety instead.
Analogous to expression for $Z_G \left( \bdmultiplication[0.5] \right)$ obtained through \cref{prop:field_theory_surfaces}\textit{(ii)}, we make the following definition.

\begin{definition}
    Denote by $m \colon G \times G \to G$ the multiplication map, and define $\mu$ to be the (exact) functor
    \[ D^b(G) \times D^b(G) \to D^b(G), \quad (\mathcal{F}, \mathcal{G}) \mapsto m_! (\pi_1^* \mathcal{F} \otimes \pi_2^* \mathcal{G}) . \]
    Since $m$ is associative, so is $\mu$, that is, $\mu(\mathcal{F}, \mu(\mathcal{G}, \mathcal{H})) = \mu(\mu(\mathcal{F}, \mathcal{G}), \mathcal{H})$ for all $\mathcal{F}, \mathcal{G}, \mathcal{H} \in D^b(G)$.
\end{definition}

The equality of bordisms $\bdmultiplication[0.5] \raisebox{0.35em}{$\circ \hspace{0.1em} \bdunit[0.5]$} = \bdidentity[0.5]$ translates to the fact that $Z_G \left( \bdmultiplication[0.5] \right) (Z_G \left( \bdunit[0.75] \right) (\uKK) \boxtimes \mathcal{F}) = \mathcal{F}$ for all $\mathcal{F} \in D(G/G)$. Actually, as $Z_G \left( \bdunit[0.75] \right)(\uKK) = (i/G)_* \uKK$ with $i \colon \{ 1 \} \to G$ the inclusion of the unit of $G$ (see \cref{prop:field_theory_surfaces}\textit{(v)}), this merely amounts to the fact that multiplication by the unit is the identity. The analogous statement for $\mu$ is given by the following proposition.

\begin{proposition}
    \label{prop:multiplication_by_one}
    $\mu(i_* \uKK, \mathcal{F}) = \mathcal{F}$ for any $\mathcal{F}$ in $D^b(G)$.
\end{proposition}
\begin{proof}
    From the cartesian square
    \[ \begin{tikzcd} \{ 1 \} \times G \arrow{r}{\pi'_1} \arrow{d}{i'} & \{ 1 \} \arrow{d}{i} \\ G \times G \arrow{r}{\pi_1} & G \end{tikzcd} \]
    follows that
    \begin{align*}
        \mu(i_* \uKK, \mathcal{F})
            &= m_!(\pi_1^* i_* \uKK \otimes \pi_2^* \mathcal{F}) \\
            &= m_!(i'_* \uKK \otimes \pi_2^* \mathcal{F}) \\
            &= m_!(i'_*(\uKK \otimes (i')^* \pi_2^* \mathcal{F})) \\
            &= m_! i'_* (i')^* \pi_2^* \mathcal{F} \\
            &= (m \circ i')_! (\pi_2 \circ i')^* \mathcal{F} \\
            &= \mathcal{F} ,
    \end{align*}
    where in the second equality we used $\pi_1^* i_* = i'_* (\pi'_1)^*$, in the third equality the projection formula, and in the fifth equality that $m \circ i'$ and $\pi_2 \circ i'$ are both equal to the isomorphism $\{ 1 \} \times G \to G$ which projects to the factor $G$.
\end{proof}

The following proposition describes multiplication with the constant sheaf on $G$.

\begin{proposition}
    \label{prop:multiplication_by_G}
    $\mu(\uKK, \mathcal{F}) = \pi^* \pi_! \mathcal{F}$ for any $\mathcal{F}$ in $D^b(G)$, where $\pi \colon G \to *$.
\end{proposition}
\begin{proof}
    From the cartesian square
    \[ \begin{tikzcd}
        G \times G \arrow{d}{m} \arrow{r}{\pi_2} & G \arrow{d}{\pi} \\ G \arrow{r}{\pi} & *
    \end{tikzcd} \]
    follows that $\mu(\uKK, \mathcal{F}) = m_! (\pi_1^* \uKK \otimes \pi_2^* \mathcal{F}) = m_! \pi_2^* \mathcal{F} = \pi^* \pi_! \mathcal{F}$.
\end{proof}

Define the \emph{commutator map}
\[ c \colon G \times G \to G, \quad (g, h) \mapsto [g, h] = g h g^{-1} h^{-1} . \]
Analogous to \eqref{eq:ZG_Sigma_g_from_convolution_mathcal_F}, the following proposition expresses the cohomology of the $G$-representation variety of $\Sigma_g$ in terms of a repeated multiplication.

\begin{proposition}
    \label{prop:formula_cohomology_compact_support_as_pullback_of_convolution}
    For any $g \ge 0$, the cohomology with compact support of the $G$-representation variety of $\Sigma_g$ is given by
    \[ H_c^*(\Rep{G}{\Sigma_g}; \KK) = i^* \big( \underbrace{c_! \uKK * \cdots * c_! \uKK}_{g \textup{ times}} \big) , \]
    where $\mu(-, -)$ is shortened to $*$.
\end{proposition}
\begin{proof}
    The statement follows from \cref{prop:proper_base_change} and the cartesian square
    \[ \begin{tikzcd}[row sep=1.25em]
        \Rep{G}{\Sigma_g} \arrow{dd} \arrow{r} & G^{2g} \arrow{d}{c^g} \\
        & G^g \arrow{d}{m_g} \\
        \{ 1 \} \arrow{r}{i} & G
    \end{tikzcd} \]
    where $m_g \colon G^g \to G$ is multiplication of $g$ elements.
\end{proof}

\section{\texorpdfstring{$\SU(2)$}{SU(2)}-representation varieties}
\label{sec:SU2_representation_varieties}

Let us now apply the above theory to compute the cohomology of the $G$-representation variety for the compact group $G = \SU(2)$ of $2 \times 2$ special unitary matrices. Explicitly, this group can be presented as
\[ \SU(2) = \left\{ \begin{pmatrix} a + bi & c + di \\ -c + di & a - bi \end{pmatrix} \;\middle|\; a^2 + b^2 + c^2 + d^2 = 1 \right\} \cong S^3 . \]
For simplicity, we will compute the cohomology with coefficients in $\KK = \QQ$, to avoid dealing with torsion. In principal, with some extra care, one should be able to upgrade the computations in this section to obtain the cohomology with coefficients in $\ZZ$. 

Recall that our strategy is to compute $c_! \uQQ \in D^b(G)$, where
\[ c \colon G \times G \to G, \quad (A, B) \mapsto [A, B] \]
denotes the commutator map. That is, we want to decompose $c_! \uQQ$ into direct summands that are easier to understand. Then, computing $\mu$ on all pairs of those summands, we can obtain an explicit expression for $c_! \uQQ * \cdots * c_! \uQQ$ in $D^b(G)$.

\subsection{The commutator map}
\label{subsection:commutator_map}

The object $\mathcal{F} = c_! \uQQ$ in $D^b(G)$ is best understood through the distinguished triangle of \cref{prop:localization_distinguished_triangle},
\begin{equation}
    \label{eq:triangle_mathcal_F}
    j_! j^* \mathcal{F} \to \mathcal{F} \to i_* i^* \mathcal{F} \xrightarrow{+}
\end{equation}
where $i \colon \{ 1 \} \to G$ denotes the inclusion of the unit of $G$, and $j \colon G \setminus \{ 1 \} \to G$ the inclusion of its open complement.

\begin{proposition}
    \label{prop:commutator_at_identity}
    $i^* c_! \uQQ = \QQ[0] \oplus \QQ[-2] \oplus \QQ[-3]^{\oplus 2}$
\end{proposition}

\begin{proof}
    From the cartesian square
    \[ \begin{tikzcd} \{ 1 \} \times_G G^2 \arrow{d}{c'} \arrow{r}{i'} & G^2 \arrow{d}{c} \\ \{ 1 \} \arrow{r}{i} & G \end{tikzcd} \]
    follows that $i^* c_! \uQQ = c'_! (i')^* \uQQ = c'_! \uQQ = c'_* \uQQ$, where the last equality is due to $c$ and $c'$ being proper. Hence, we must compute the cohomology of the fiber
    \[ X \defequality \{ 1 \} \times_G G^2 \cong \big\{ (A, B) \in G^2 \mid AB = BA \big\} . \]
    There are many ways to do so. Note that any $A \in \SU(2)$ is diagonalizable and can be written as $A = P \smatrix{\alpha & 0 \\ 0 & \overline{\alpha}} P^{-1}$ for some $\alpha \in U(1)$ and $P \in \SU(2)$. When $A \ne \pm 1$, this expression is unique up to the transformation $(\alpha, P) \mapsto \big(\overline{\alpha}, P \smatrix{0 & 1 \\ 1 & 0}\big)$. Furthermore, the centralizer of any $A \ne \pm 1$ in $\SU(2)$ is the diagonal subgroup $\U(1) \subset \SU(2)$. In particular, we obtain a morphism
    \[ f \colon \big(\SU(2) / \U(1) \times \U(1)^2 \big) / S_2 \to X, \quad (P, \alpha, \beta) \mapsto (A = P \smatrix{\alpha & 0 \\ 0 & \overline{\alpha}} P^{-1}, B = P \smatrix{\beta & 0 \\ 0 & \overline{\beta}} P^{-1}) , \]
    where the group $S_2$ of order two acts on $\U(1)$ by complex conjugation, and on $\SU(2)/\U(1)$ by right multiplication with $\smatrix{0 & 1 \\ 1 & 0}$. Note that the fibers of $f$ are a single point, except for the fibers over $(A, B) = (\pm 1, \pm 1)$, where the fibers are $(\SU(2)/\U(1)) / S_2 \cong S^2 / S_2 \cong \mathbb{RP}^2$ (indeed, $S_2$ acts on $\SU(2)/\U(1) \cong S^2$ via the antipodal map, so the quotient is the real projective plane).
    This shows that the natural map $\uQQ \to f_* f^* \uQQ$ is an isomorphism: for any $x \in X$, the map on stalks $H^*(\{ x \}; \QQ) = (\uQQ)_x \to (f_* f^* \uQQ)_x = H^*(f^{-1}(\{ x \}); \QQ)$ is an isomorphism. Hence, $c'_* \uQQ = c'_* f_* \uQQ$, so it suffices to compute the cohomology of the domain of $f$ instead (this trick is known as the \textit{Vietoris--Begle theorem}, cf.\ \cite[Corollary 2.7.7]{KashiwaraSchapira1990}).

    Write $Y \defequality \SU(2) / \U(1) \times \U(1)^2$. To compute the cohomology of the domain of $f$, that is, of $Y / S_2$, we use \cite[Theorem II.19.2]{Bredon1997}, which states that the cohomology $H^*(Y/S_2; \QQ)$ of the quotient $Y/S_2$ is equal to the $S_2$-invariant subspace $H^*(Y; \QQ)^{S_2}$ of the cohomology of $Y$. 
    Since $S_2$ acts on $\U(1)$ by complex conjugation, it acts trivially on $H^0(\U(1); \QQ) = \QQ$ and non-trivially on $H^1(\U(1); \QQ) = \QQ$. Similarly, $S_2$ acts on $\SU(2)/\U(1) \cong S^2$ via the antipodal map, so it acts trivially on $H^0(\SU(2)/\U(1); \QQ) = \QQ$ and non-trivially on $H^2(\SU(2)/\U(1); \QQ) = \QQ$. It follows that
    \[ H^*(Y/S_2; \QQ) = \QQ[0] \oplus \QQ[2] \oplus \QQ[3]^{\oplus 2} . \]
\end{proof}

\begin{proposition}
    \label{prop:commutator_at_general}
    $j^* c_! \uQQ = \uQQ[0] \oplus \uQQ[3]$
\end{proposition}
\begin{proof}
    From the cartesian square
    \[ \begin{tikzcd} U \times_G G^2 \arrow{d}{c'} \arrow{r}{j'} & G^2 \arrow{d}{c} \\ U \arrow{r}{j} & G \end{tikzcd} \]
    follows that $j^* c_! \uQQ = c'_! (j')^* \uQQ = c'_! \uQQ$.
    We will show that $U \times_G G^2$ is isomorphic to the trivial fiber bundle $U \times \SO(3)$ over $U$, so that $c'_* \uQQ = \bigoplus_{n \ge 0} \uQQ \boxtimes H^n(\SO(3), \uQQ) = \uQQ[0] \oplus \uQQ[-3]$ as desired.

    Since $U \cong \RR^3$ is contractible, it suffices to show that $c' \colon U \times_G G^2 \to U$ is a fiber bundle with fiber $\SO(3)$.
    Since this map is equivariant with respect to the action of conjugation by $G$, it suffices to look at the fibers $\big\{ (A, B) \in G^2 \mid AB = gBA \big\}$ over the diagonal matrices $g = -\exp \smatrix{i \theta & 0 \\ 0 & - i \theta}$ with $\theta \in [0, \pi)$.
    Using a substitution of variables $A = A' \exp \smatrix{i \theta / 2 & 0 \\ 0 & - i \theta / 2}$ and $B = \exp \smatrix{-i \theta / 2 & 0 \\ 0 & i \theta / 2} B'$, the equation $AB = gBA$ reduces to
    \[ A' B' + \exp \smatrix{i \theta / 2 & 0 \\ 0 & - i \theta / 2} B' A' \exp \smatrix{i \theta / 2 & 0 \\ 0 & - i \theta / 2} = 0 . \]
    Looking at the coefficients of this matrix equation, we find
    \[ \begin{gathered}
        a x \cos \theta + a x - a y \sin \theta - b x \sin \theta - b y \cos \theta - b y \\
        + c w \sin \theta - c z \cos \theta - c z - d w \cos \theta - d w - d z \sin \theta = 0, \\
        ay + bx = 0 , \\
        az + cx = 0 , \\
        aw + dx = 0 .
    \end{gathered} \]
    Note that, if $a \ne 0$, we can solve for $(y, z, w) = - \tfrac{x}{a} (b, c, d)$ using the last three equations, after which the first equation reduces to $\tfrac{x}{a} (1 + \cos \theta) = 0$, that is, $x = 0$. But then $x = y = z = w = 0$, which is not a valid solution, so we must have $a = 0$. Since $(b, c, d) \ne (0, 0, 0)$, it follows from the last three equations that $x = 0$ as well.
    The first equation now reads
    \[ b y \cos \theta + b y - c w \sin \theta + c z \cos \theta + c z + d w \cos \theta + d w + d z \sin \theta = 0 . \]
    After a rotation of variables $\smatrix{z \\ w} = \smatrix{\cos \theta / 2 & \sin \theta / 2 \\ -\sin \theta / 2 & \cos \theta / 2} \smatrix{z' \\ w'}$, this equation reduces to
    \[ by (1 + \cos \theta) + 2 cz' \cos \tfrac{\theta}{2} + 2 dw' \cos \tfrac{\theta}{2} = 0 . \]
    Since $1 + \cos \theta$ and $\cos \tfrac{\theta}{2}$ are both positive for $\theta \in [0, \pi)$, this space is homeomorphic to the unit tangent bundle of $S^2$. Finally, the unit tangent bundle of $S^2$ is homeomorphic to $\SO(3)$ as both of them parametrize all oriented orthonormal frames of $\RR^3$. We conclude that $U \times_G G^2 \to U$ is indeed a fiber bundle with fiber $\SO(3)$.
\end{proof}

Returning to the distinguished triangle \eqref{eq:triangle_mathcal_F}, to properly understand $c_! \uQQ$ in terms of $i_* i^* c_! \uQQ$ and $j_! j^* c_! \uQQ$, we must understand the connecting morphism $i_* i^* c_! \uQQ \to j_! j^* c_! \uQQ[1]$. In particular, we must understand the Ext groups $\Ext_{\Sh(G)}^n(i_* \uQQ, j_! \uQQ) = \Hom_{D^b(G)}(i_* \uQQ, j_! \uQQ[n])$ between the `extension by zero' sheaves $i_* \uQQ$ and $j_! \uQQ$. Recall that these sheaves fit into a short exact sequence
\begin{equation}
    \label{eq:short_exact_sequence}
    0 \to j_! \uQQ \to \uQQ \to i_* \uQQ \to 0 .
\end{equation}

\begin{proposition}
    \label{prop:ext_groups_extensions_by_zero}
    For any $n \in \ZZ$, we have
    \[ \Hom_{D^b(G)}(i_* \uQQ, j_! \uQQ[n]) = \Ext_{\Sh(G)}^n(i_* \uQQ, j_! \uQQ[n]) = \begin{cases} \QQ & \textup{if } n \in \{ 1, 3 \}, \\ 0 & \textup{otherwise}. \end{cases} \]
\end{proposition}
\begin{proof}
    Apply $\Hom_{\Sh(G)}(-, j_! \uQQ)$ to \eqref{eq:short_exact_sequence} to obtain a long exact sequence
    \[ \cdots \gets \Ext^n_{\Sh(G)}(j_! \uQQ, j_! \uQQ) \gets \Ext^n_{\Sh(G)}(\uQQ, j_! \uQQ) \gets \Ext^n_{\Sh(G)}(i_* \uQQ, j_! \uQQ) \gets \Ext^{n - 1}_{\Sh(G)}(j_! \uQQ, j_! \uQQ) \gets \cdots \]
    Using the adjunction $j_! \dashv j^*$, we find that
    \[ \Ext^n_{\Sh(G)}(j_! \uQQ, j_! \uQQ) = \Ext^n_{\Sh(U)}(\uQQ, \uQQ) = H^n(U; \QQ) = \begin{cases} \QQ & \textup{if } n = 0 , \\ 0 & \textup{otherwise}. \end{cases} \]
    Furthermore, since $\Hom_{\Sh(G)}(\uQQ, -) = \pi_*$ and $U \cong \RR^3$, we have
    \[ \Ext^n_{\Sh(G)}(\uQQ, j_! \uQQ) = H^n_c(U; \QQ) = \begin{cases} \QQ & \textup{if } n = 3 , \\ 0 & \textup{otherwise}. \end{cases} \]
    The result now follows from the long exact sequence.
\end{proof}

\begin{remark}
    \label{rem:delta_epsilon}
    The non-zero element in $\Ext_{\Sh(G)}^1(i_* \uQQ, j_! \uQQ) = \QQ$, which we will denote by $\delta$, corresponds to the extension \eqref{eq:short_exact_sequence}. Denote by $\varepsilon$ the non-zero element in $\Ext_{\Sh(G)}^3(i_* \uQQ, j_! \uQQ) = \QQ$. Note that, ultimately, the element $\varepsilon$ corresponds to the non-trivial cocycle in $H^3(G; \QQ) = \QQ$.
\end{remark}

\begin{proposition}
    \label{prop:pushforward_along_commutator}
    $c_! \uQQ = i_* \uQQ[-3]^{\oplus 2} \oplus j_! \uQQ[-3] \oplus \Cone\left(i_* \uQQ[-1] \oplus i_* \uQQ[-3] \xrightarrow{(\delta \; \varepsilon)} j_! \uQQ[0] \right)$
\end{proposition}
\begin{proof}
    Write $\mathcal{F} = c_! \uQQ$. The distinguished triangle \eqref{eq:triangle_mathcal_F} reduces, by \cref{prop:commutator_at_identity,prop:commutator_at_general}, after a shift to
    \[ i_* \uQQ[-1] \oplus i_* \uQQ[-3] \oplus i_* \uQQ[-4]^{\oplus 2} \xrightarrow{\varphi} j_! \uQQ[0] \oplus j_! \uQQ[-3] \to \mathcal{F} \xrightarrow{+} . \]
     It follows from \cref{prop:ext_groups_extensions_by_zero} that $\varphi$ must be of the form 
    \[ \varphi = \begin{pmatrix} \varphi_{11} & \varphi_{12} & 0 & 0 \\ 0 & 0 & \varphi_{23} & \varphi_{24} \end{pmatrix} . \]
    Note that $H^0(c_! \uQQ)$ is equal to the underived direct image $c_! \uQQ$, which is $\uQQ$, and hence $\varphi_{11} = \delta$. Since $H^2(\pi_* \mathcal{F}) = H^2(G^2; \QQ) = 0$, where $\pi \colon G \to *$ is the projection to a point, we must have $\varphi_{12} \ne 0$, so $\varphi_{12} = \varepsilon$.
    Finally, we claim that $\varphi_{23} = \varphi_{24} = 0$, which we show using the equivariant case. Following the proof of \cref{prop:commutator_at_general}, we see that $j^* \mathcal{F}$ corresponds to the derived direct image $(U \times \SO(3) \to G)_! \uQQ$. Note that $G$ acts on $\SO(3)$ by left multiplication (via the double cover $G = \SU(2) \xrightarrow{\rho} \SO(3)$), so that $j^* \mathcal{F}$ may also be identified with the pullback of $(U / \{ \pm 1 \} \xrightarrow{h} G \xrightarrow{q_G} G/G)_! \uQQ$ (where $\{ \pm 1 \} = \ker \rho$ acts trivially on $U$) along the quotient map $q_G \colon G \to G/G$. In particular, the morphisms $\varphi_{23}, \varphi_{24}$ can be identified with pullbacks of some morphisms $(i/G)_* \uQQ[-4] \to (q_G)_* h_! \uQQ$ along $q_G$. But, similar to \cref{prop:ext_groups_extensions_by_zero}, one shows that $\Hom_{D(G/G)}((i/G)_* \uQQ[-4], (q_G)_* h_! \uQQ) = \Hom_{D^b(G)}(i_* \uQQ, h_! \uQQ[4]) = 0$.
\end{proof}

\subsection{The multiplication map}
\label{subsection:multiplication_map}

For convenience, we will denote the $\Cone\Big(i_* \uQQ[-1] \oplus i_* \uQQ[-3] \xrightarrow{(\delta \; \varepsilon)} j_! \uQQ[0] \Big)$ by $\mathcal{E}$, so that $\mathcal{F} = c_! \uQQ = i_* \uQQ[3]^{\oplus 2} \oplus j_!\uQQ[-3] \oplus \mathcal{E}$.
Now, in order to compute the product $\mathcal{F} * \cdots * \mathcal{F}$, we must understand the products
\[ \mu(i_* \uQQ, i_* \uQQ), \quad \mu(i_* \uQQ, j_! \uQQ), \quad \mu(i_* \uQQ, \mathcal{E}) , \quad \mu(j_! \uQQ, j_! \uQQ), \quad \mu(j_! \uQQ, \mathcal{E}) \quad \textup{and} \quad \mu(\mathcal{E}, \mathcal{E}) . \]
For the first three of these, we can use \cref{prop:multiplication_by_one} (replacing the coefficients from $\ZZ$ to $\QQ$). To provide a good description for the fourth, we make the following definition.

\begin{definition}
    Define a sequence of objects $\mathcal{U}_n \in D^b(G)$ as follows. Put $\mathcal{U}_0 = i_* \uQQ$ and inductively define $\mathcal{U}_n = \mu(j_! \uQQ, \mathcal{U}_{n - 1})$ for all $n \ge 1$. In particular, $\mathcal{U}_1 = j_! \uQQ$.
\end{definition}

\begin{proposition}
    \label{prop:properties_mathcal_U}
    For every $n \ge 1$, we have
    \begin{enumerate}[label=(\roman*)]
        \item $\pi_* \mathcal{U}_n = \QQ[-3n]$
        \item $i^* \mathcal{U}_n = \bigoplus_{k = 0}^{n - 2} \QQ[-2k - n - 1]$
    \end{enumerate}
\end{proposition}
\begin{proof}
    Statement \textit{(i)} follows from the fact that $\pi_* \mathcal{U}_n = H^*_c(U^n; \QQ) = \QQ[-3n]$. We prove \textit{(ii)} by induction on $n$, the case $n = 1$ being $i^* \mathcal{U}_1 = i^* j_! \uQQ = 0$. For $n \ge 2$, we can understand $\mathcal{U}_n = \mu(j_! \uQQ, \mathcal{U}_{n - 1})$ via the triangle obtained by applying $\mu(-, \mathcal{U}_{n - 1})$ to \eqref{eq:short_exact_sequence}:
    \[ \begin{tikzcd}[row sep=0.5em]
        \mu(\uQQ, \mathcal{U}_{n - 1})[-1] \arrow{r} & \mu(i_* \uQQ, \mathcal{U}_{n - 1})[-1] \arrow{r} & \mu(j_! \uQQ, \mathcal{U}_{n - 1}) \arrow{r}{+} & \hphantom{} \\
        \pi^* \pi_* \mathcal{U}_{n - 1} = \uQQ[-3 n + 2] \arrow[equals]{u} & \mathcal{U}_{n - 1}[-1] \arrow[equals]{u} & \mathcal{U}_n \arrow[equals]{u} &
    \end{tikzcd} \]
    Applying $i^*$ to this triangle, we find that
    \[ i^* \mathcal{U}_n = \Cone \left( \QQ[-3n + 2] \to \bigoplus_{k = 0}^{n - 3} \QQ[-2k - n - 1] \right) = \bigoplus_{k = 0}^{n - 2} \QQ[-2k - n - 1] . \qedhere \]
\end{proof}

\begin{definition}
    Define a sequence of objects $\mathcal{E}_n \in D^b(G)$ as follows. Put $\mathcal{E}_0 = i_* \uQQ[0]$ and inductively define $\mathcal{E}_n = \mu(\mathcal{E}, \mathcal{E}_{n - 1})$ for all $n \ge 1$. In particular, $\mathcal{E}_1 = \mathcal{E}$.
\end{definition}


\begin{proposition}
    \label{prop:properties_mathcal_E}
    For every $n \ge 0$, we have
    \begin{enumerate}[label=(\roman*)]
        \item $\pi_* \mathcal{E}_n = \QQ[0]$ 
        \item $i^* \mathcal{E}_n = \bigoplus_{k = 0}^{n} \QQ[-2k]$
    \end{enumerate}
\end{proposition}
\begin{proof}
    We prove these statements by induction on $n$. For $n = 0$, they follow from the definition $\mathcal{E}_0 = i_* \uQQ$. Let us now prove the statements for $n \ge 1$ assuming the statements hold for $n - 1$. Apply $\mu(-, \mathcal{E}_{n - 1})$ to the triangle
    \begin{equation}
        \label{eq:triangle_for_mathcal_E}
        i_* \uQQ[-3] \xrightarrow{\varepsilon} \uQQ \to \mathcal{E} \xrightarrow{+}
    \end{equation}
    to obtain, using \cref{prop:multiplication_by_one} and \cref{prop:multiplication_by_G}, the distinguished triangle
    \[ \begin{tikzcd}[row sep=0.5em]
        \mu(i_* \uQQ[-3], \mathcal{E}_{n - 1}) \arrow{r} & \mu(\uQQ, \mathcal{E}_{n - 1}) \arrow{r} & \mu(\mathcal{E}, \mathcal{E}_{n - 1}) \arrow{r}{+} & \hphantom{} \\
        \mathcal{E}_{n - 1}[-3] \arrow[equals]{u} & \pi^* \pi_* \mathcal{E}_{n - 1} = \uQQ \arrow[equals]{u} & \mathcal{E}_n \arrow[equals]{u} &
    \end{tikzcd} \]
    Applying $\pi_*$ to this triangle, we find 
    \[ \pi_* \mathcal{E}_n = \Cone \left( \QQ[-3] \xrightarrow{\smatrix{0 \\ 1}} \QQ[0] \oplus \QQ[-3] \right) = \QQ[0] . \]
    Similarly, applying $i^*$ to the triangle gives
    \[ i^* \mathcal{E}_n = \Cone \left( \bigoplus_{k = 0}^{n - 1} \QQ[-2k - 3] \to \QQ[0] \right) = \bigoplus_{k = 0}^{n} \QQ[-2k] . \qedhere \]
\end{proof}

\begin{lemma}
    \label{lemma:mu_mathcal_E_U}
    $\mu(\mathcal{E}, j_! \uQQ) = i_* \uQQ[-3]$
\end{lemma}
\begin{proof}
    Apply $\mu(-, j_! \uQQ)$ to the triangle \eqref{eq:triangle_for_mathcal_E} to find, using \cref{prop:multiplication_by_one,prop:multiplication_by_G}, that
    \[ \begin{tikzcd}[row sep=0.5em]
        \mu(i_* \uQQ[-3], j_! \uQQ) \arrow{r}{\alpha} & \mu(\uQQ, j_! \uQQ) \arrow{r} & \mu(\mathcal{E}, j_! \uQQ) \arrow{r}{+} & \hphantom{} \\
        j_! \uQQ[-3] \arrow[equals]{u} & \pi^* \pi_* j_! \uQQ = \uQQ[-3] \arrow[equals]{u} & &
    \end{tikzcd} \]
    Since $\pi_* \mu(\mathcal{E}, j_! \uQQ) = \pi_* \mathcal{E} \otimes \pi_* j_! \uQQ = \QQ[-3]$, the morphism $\alpha \colon j_! \uQQ[-3] \to \uQQ[-3]$ must be a non-trivial element in $\Hom_{\Sh(G)}(j_! \uQQ, \uQQ) = \Hom_{\Sh(U)}(\uQQ, \uQQ) = H^0(U; \QQ) = \QQ$, which corresponds to the non-trivial extension \eqref{eq:short_exact_sequence}. Hence, $\mu(\mathcal{E}, j_! \uQQ) = \Cone \big( j_! \uQQ[-3] \xrightarrow{\alpha} \uQQ[-3] \big) = i_* \uQQ[-3]$ as desired.
\end{proof}

\begin{definition}
    Define a sequence of objects $\mathcal{F}_n \in D^b(G)$ as follows. Put $\mathcal{F}_0 = i_* \uQQ[0]$ and inductively define $\mathcal{F}_n = \mu(\mathcal{F}, \mathcal{F}_{n - 1})$ for all $n \ge 1$. In particular, $\mathcal{F}_1 = \mathcal{F}$. Note that, by construction, we have $\mathcal{F}_n = (c_n)_! \uQQ$ for $c_n \colon G^{2n} \to G$ given by $(A_1, B_1, \ldots, A_n, B_n) \mapsto [A_1, B_1] \cdots [A_n, B_n]$.
\end{definition}

\begin{proposition}
    \label{prop:explicit_expression_mathcal_Fn}
    For any $n \ge 1$, we have
    \[ \mathcal{F}_n = i_* \uQQ[-3n]^{\oplus \binom{2n}{n}} \oplus \bigoplus_{k = 1}^{n} \mathcal{U}_k[-3n]^{\oplus \binom{2n}{n - k}} \oplus \bigoplus_{k = 1}^{n} \mathcal{E}_k[-3(n - k)]^{\oplus \binom{2n}{n - k}} . \]
\end{proposition}
\begin{proof}
    Proof by induction on $n$, the case $n = 1$ being \cref{prop:pushforward_along_commutator}. For $n > 1$, one computes $\mathcal{F}_n = \mu(\mathcal{F}, \mathcal{F}_{n - 1}) = \mu(i_* \uQQ[-3], \mathcal{F}_{n - 1})^{\oplus 2} \oplus \mu(j_! \uQQ[-3], \mathcal{F}_{n - 1}) \oplus \mu(\mathcal{E}, \mathcal{F}_{n - 1})$ using \cref{prop:multiplication_by_one,prop:properties_mathcal_U,prop:properties_mathcal_E,lemma:mu_mathcal_E_U} and the induction hypothesis.
\end{proof}

\begin{remark}
    As a sanity check, one can verify that $\pi_* \mathcal{F}_n = (\QQ[0] \oplus \QQ[-3])^{\otimes 2n}$ equals $H^*(G^{2n}; \QQ)$, using \cref{prop:properties_mathcal_U,prop:properties_mathcal_E}, as expected.
\end{remark}

\begin{theorem}
    \label{thm:poincare_polynomial_SU2_representation_varieties}
    The Poincaré polynomial of the $\SU(2)$-representation variety of $\Sigma_g$ is given by
    \[ P(\Rep{\SU(2)}{\Sigma_g}) = \binom{2g}{g} t^{3g} + \sum_{k = 1}^{g} \binom{2g}{g - k} \frac{t^{2 - k} - t^{k + 1} + t^{3k - 1} - t^{-3k}}{t^2 - 1} . \]
\end{theorem}
\begin{proof}
    By \cref{prop:formula_cohomology_compact_support_as_pullback_of_convolution}, the cohomology of the $\SU(2)$-representation variety of $\Sigma_g$ (which is equal to the cohomology with compact support as $G$ is compact) is given by $i^* \mathcal{F}_g$. Hence, the result follows from \cref{prop:explicit_expression_mathcal_Fn} using that $i^* i_* \uQQ = \QQ$ and $i^* \mathcal{U}_n = \bigoplus_{k = 0}^{n - 2} \QQ[-2k - n - 1]$ and $i^* \mathcal{E}_n = \bigoplus_{k = 0}^{n} \QQ[-2k]$.
\end{proof}

\begin{example}
    The first few Poincaré polynomials are
    \begin{align*}
        P(\Rep{\SU(2)}{\Sigma_0}) &= 1 , \\
        P(\Rep{\SU(2)}{\Sigma_1}) &= 2 t^3 + t^2 + 1 , \\
        P(\Rep{\SU(2)}{\Sigma_2}) &= t^9 + 6 t^6 + 4 t^5 + t^4 + 4 t^3 + t^2 + 1 .
    \end{align*}
\end{example}

\begin{remark}
    A recurring theme in the computation of cohomological invariants of representation varieties of $\Sigma_g$ in the $\K$-theoretic setting, such as in \cite{GonzalezLogaresMunoz2020, Gonzalez2020, HablicsekVogel2022}, is that it suffices to work within a finitely generated submodule of the $\K$-group over $G$ (e.g.\ $\K(\cat{MHM}_G)$ or $\K(\Var_G)$ or similar). More precisely, one computes the class $[c]$ of the commutator map $c \colon G^2 \to G$ in the $\K$-group over $G$ and that of the iterated products $[c] * \cdots * [c]$, and although it has not been proven in any level of generality, in practice it turns out that these classes are contained within a finitely generated submodule of the $\K$-group over $G$. Consequently, the operation $[c] * (-)$ can be expressed as a matrix with respect to these generators, from which follows a recurrence relation between the $\K$-invariants of the representation varieties of $\Sigma_g$ for increasing genus $g$. In particular, the $\K$-invariants for all $g$ can be determined from the $\K$-invariants of sufficiently finitely many values of $g$.

    Now, this phenomenon seems to have disappeared in the categorified setting, as we encounter infinitely many non-trivial extensions $\mathcal{U}_n$ and $\mathcal{E}_n$, which cannot be expressed in terms of each other. Consequently, there is no recurrence relation between the polynomials $P(\Rep{\SU(2)}{\Sigma_g})$. On the other hand, when passing to the $\K$-theoretic setting, we find that $[\mathcal{E}_n] = [\mathcal{E}_{n - 1}] + [\uQQ]$ and $[\mathcal{U}_n] = (-1)^{n + 1} [\uQQ] - [\mathcal{U}_{n - 1}]$ in $\K(D^b(G))$ for all $n \ge 1$, which shows that the computation for the $\K$-invariant does restrict to a finitely generated submodule of $\K(D^b(G))$.
\end{remark}

\subsection{Twisted \texorpdfstring{$\SU(2)$}{SU(2)}-representation varieties}

Extracting the cohomology of the representation variety out of the explicit expression for $\mathcal{F}_n$ is a great result. However, the explicit expression in \cref{prop:explicit_expression_mathcal_Fn} contains much more information than this. For example, for any element $C \in G \setminus \{ 1 \}$ we can compute $i_C^* \mathcal{F}_g$ for the inclusion $i_C \colon \{ C \} \to G$. This way, we arrive at the notion of the \emph{twisted representation variety}.

\begin{definition}
    For any $C \in G \setminus \{ 1 \}$, the \emph{twisted $G$-representation variety} of $\Sigma_g$ with respect to $C$ is given by
    \[ \RepTw{G}{\Sigma_g}{C} = \left\{ (A_1, B_1, \ldots, A_g, B_g) \in G^{2g} \;\middle|\; [A_1, B_1] \cdots [A_g, B_g] = C \right\} . \]
    Similarly, the \emph{twisted $G$-character stack} of $\Sigma_g$ with respect to $C$ is the quotient stack
    \[ \CharStckTw{G}{\Sigma_g}{C} = \RepTw{G}{\Sigma_g}{C} / G \]
    where $G$ acts on $\RepTw{G}{\Sigma_g}{C}$ by conjugation.
\end{definition}

\begin{remark}
    Elements of the twisted $G$-representation variety $\RepTw{G}{\Sigma_g}{C}$ can be interpreted as representations from the fundamental group $\pi_1(\Sigma_g \setminus \{ x \}, *)$ of the punctured surface $\Sigma_g \setminus \{ x \}$ for some point $x \in \Sigma_g$, such that a small loop around $x$ is sent to $C$.
\end{remark}

\begin{corollary}
    \label{cor:poincare_polynomial_twisted_SU2_representation_varieties}
    The Poincaré polynomial of the twisted $\SU(2)$-representation variety of $\Sigma_g$ is given by
    \[ P(\RepTw{\SU(2)}{\Sigma_g}{C}) = t^{3n - 3/2} \sum_{k = 1}^{n} \binom{2n}{n - k} \frac{(t^k - t^{-k}) (t^{2k - 1/2} + t^{-2k + 1/2})}{t - t^{-1}} \]
    for any $C \in \SU(2) \setminus \{ 1 \}$. In particular, it is symmetric.
\end{corollary}
\begin{proof}
    The statement follows from \cref{prop:explicit_expression_mathcal_Fn}, using the facts that $i_C^* i_* \uQQ = 0$ and $i_C^* \mathcal{U}_n = \bigoplus_{k = 0}^{n - 1} \QQ[-2k - n + 1]$ and $i_C^* \mathcal{E}_n = \bigoplus_{k = 0}^{n - 1} \QQ[-2k]$ for $n \ge 1$. The last two equalities are shown similar to \cref{prop:properties_mathcal_E} \textit{(ii)} and \cref{prop:properties_mathcal_U} \textit{(ii)}, respectively.
\end{proof}

\begin{remark}
    The twisted representation varieties are smooth \cite[Theorem 2.2.5]{HauselRodriguezVillegas2008} and compact. Assuming orientability, this explains why their Poincaré polynomials are symmetric: due to Poincaré duality.
\end{remark}

\subsection{Non-orientable surfaces}

The method used in \cref{subsection:commutator_map,subsection:multiplication_map} to compute the Poincaré polynomials of the representation varieties of closed orientable surfaces works equally well for non-orientable surfaces. Denote by $N_r$ the non-orientable surface of demigenus $r$, that is, of Euler characteristic $2 - r$. Equivalently, $N_r$ is the connected sum of $r$ real projective planes. The fundamental group of $N_r$ is given by
\[ \pi_1(N_r, *) = \langle a_1, \ldots, a_r \mid a_1^2 \cdots a_r^2 = 1 \rangle . \]
Hence, to study the corresponding representation variety $\Rep{\SU(2)}{N_r}$, we turn our attention to the squaring map $A \mapsto A^2$ rather than the commutator map $(A, B) \mapsto [A, B]$. Actually, it turns out to be convenient to study the squaring map with a minus sign
\[ s \colon G \to G, \quad A \mapsto -A^2 \]
instead. Then, the cohomology of $\Rep{\SU(2)}{N_r}$ is given by
\begin{equation}
    \label{eq:cohomology_non_orientable_surfaces}
    H^*(\Rep{\SU(2)}{N_r}; \QQ) = \begin{cases}
        i^* (\underbrace{s_! \uQQ * \cdots * s_! \uQQ}_{r \textup{ times}}) & \textup{ if } r \textup{ is even} , \\
        i_{-1}^* (\underbrace{s_! \uQQ * \cdots * s_! \uQQ}_{r \textup{ times}}) & \textup{ if } r \textup{ is odd} .
    \end{cases}
\end{equation}

\begin{proposition}
    $s_! \uQQ = j_! \uQQ \oplus \mathcal{E}$
\end{proposition}
\begin{proof}
    Write $\mathcal{S} = s_! \uQQ$. By \cref{prop:localization_distinguished_triangle}, we have the distinguished triangle
    \[ j_! j^* \mathcal{S} \to \mathcal{S} \to i_* i^* \mathcal{S} \xrightarrow{+} . \]
    Note that $i^* \mathcal{S}$ is the cohomology of the trace-zero matrices, that is, $S^2$, so $i^* \mathcal{S} = \QQ[0] \oplus \QQ[-2]$. Furthermore, $j^* \mathcal{S}$ is just two copies of $U$, that is, $j^* \mathcal{S} = \uQQ^{\oplus 2}$. Hence,
    \[ \mathcal{S} = \Cone\big(i_* \uQQ[-1] \oplus i_* \uQQ[-2] \xrightarrow{\varphi} j_! \uQQ[0] \oplus j_! \uQQ[0] \big) . \]
    Due to the constraint $\pi_* \mathcal{S} = H^*(G; \QQ) = \QQ[0] \oplus \QQ[3]$, we can only have $\varphi = \smatrix{\delta & \varepsilon \\ \delta & \varepsilon}$.
\end{proof}

\begin{proposition}
    \label{prop:expression_Sn}
    Let $\mathcal{S}_0 = i_* \uQQ$ and inductively define $\mathcal{S}_n = \mu(s_! \uQQ, \mathcal{S}_{n - 1})$ for $n \ge 1$. Then for any $n \ge 1$, we have
    \[ \mathcal{S}_n = \bigoplus_{k = 0}^{\lfloor (n - 1) / 2 \rfloor} \mathcal{U}_{n - 2k}[-3k]^{\oplus \binom{n}{k}} \oplus \bigoplus_{k = 0}^{\lfloor n / 2 \rfloor} \mathcal{E}_{n - 2k}[-3k]^{\oplus \binom{n}{k}} . \]
\end{proposition}

The Poincaré polynomial of $\Rep{\SU(2)}{N_r}$ can now be computed by applying $i^*$ (resp.\ $i_{-1}^*$) to the expression for $\mathcal{S}_r$ when $r$ is even (resp.\ when $r$ is odd), according to \eqref{eq:cohomology_non_orientable_surfaces}.

\subsection{\texorpdfstring{$\SO(3)$}{SO(3)}- and \texorpdfstring{$\U(2)$}{U(2)}-representation varieties}

In this section, we will show how the computations for the $\SU(2)$-representation varieties of $\Sigma_g$ can be adapted to obtain the cohomology of the $\SO(3)$- and $\U(2)$-representation varieties of $\Sigma_g$.

Observe that, since $\SO(3)$ is the quotient of $\SU(2)$ by its center $\{ \pm 1 \}$, the commutator map $c' \colon \SO(3)^2 \to \SO(3)$ factors through $\SU(2)$: indeed, the commutator $[A, B]$ is invariant under sign changes $A \mapsto -A$ or $B \mapsto -B$. Writing $\overline{c} \colon \SO(3)^2 \to \SU(2)$ for the lift of the commutator map, we obtain the following commutative diagram.
\[ \begin{tikzcd}
    \SU(2)^2 \arrow{r} \arrow{d}{c} & \SO(3)^2 \arrow{d}{c'} \arrow[dashed, swap]{ld}{\overline{c}} \\
    \SU(2) \arrow{r} & \SO(3)
\end{tikzcd} \]
This shows that the computation for the cohomology of $\Rep{\SO(3)}{\Sigma_g}$ can be performed in $D^b(\SU(2))$, which allows us to re-use the tools developed in \cref{subsection:commutator_map,subsection:multiplication_map}. However, note that, at the end of the computation, instead of pulling back along $\{ 1 \} \to \SO(3)$, one must pull back along $\{ \pm 1 \} \to \SU(2)$, because the pre-image of $\{ 1 \}$ along $\SU(2) \to \SO(3)$ is $\{ \pm 1 \}$. In particular, $\Rep{\SO(3)}{\Sigma_g}$ consists of two connected components.

\begin{proposition}
    $\overline{c}_! \uQQ = c_! \uQQ$
\end{proposition}
\begin{proof}
    Analogous to the proof of \cref{prop:commutator_at_identity}, we can identify $i^* \overline{c}_! \uQQ$ with the cohomology of
    \[ (\SU(2)/\U(1) \times \U(1)^2) / S_2 / \{ \pm 1 \}^2 \]
    where the two copies of $\{ \pm 1 \}$ act on the two copies of $\U(1)$ by negation, respectively. Note that the action of $\{ \pm 1 \}^2$ acts trivially on the cohomology of $\U(1)^2$, so it follows that $i^* \overline{c}_! \uQQ$ agrees with the cohomology of $(\SU(2)/\U(1) \times \U(1)^2) / S_2$, that is, agrees with $i^* c_! \uQQ$.
    
    Analogous to the proof of \cref{prop:commutator_at_general}, the pullback $U \times_{\SU(2)} \SO(3)^2 = \overline{c}^{-1}(U)$ is a trivial fiber bundle over $U$, whose fibers are $\overline{c}^{-1}(\{ -1 \}) = c^{-1}(\{ -1 \}) / \{ \pm 1 \}^2$. Recall that $c^{-1}(\{ -1 \}) \cong \SO(3)$, and note that $\{ \pm 1 \}^2$ acts on $\SO(3)$ via left multiplication by the subgroup generated by $\operatorname{diag}(-1, 1, -1)$ and $\operatorname{diag}(1, -1, -1)$. Since this action can be extended to a continuous action of $\SO(3)$, the induced action of $\{ \pm 1 \}^2$ on the cohomology of $c^{-1}(\{ -1 \})$ is trivial, and therefore $\overline{c}^{-1}(\{ -1 \})$ has the same cohomology as $c^{-1}(\{ -1 \})$. Hence, $j^* \overline{c}_! \uQQ = j^* c_! \uQQ$.
    
    Finally, one repeats the proof of \cref{prop:pushforward_along_commutator} to arrive at $\overline{c}_! \uQQ = c_! \uQQ$, as desired.
\end{proof}

\begin{corollary}
    \label{cor:poincare_polynomial_SO3_representation_varieties}
    For any $g \ge 1$, the $\SO(3)$-representation variety of $\Sigma_g$ consists of two connected components, whose Poincaré polynomials are equal to the Poincaré polynomial of the $\SU(2)$-representation variety and the Poincaré polynomial of the twisted $\SU(2)$-representation variety, respectively. \qed
\end{corollary}

Next, let us consider the $\U(2)$-representation varieties of $\Sigma_g$. Note that the commutator of matrices in $\U(2)$ always lies in $\SU(2)$, so again we can use the tools developed for $D^b(\SU(2))$.
Moreover, the following argument shows how one can express the cohomology of $\Rep{\U(n)}{\Sigma_g}$ in terms of that of $\Rep{\PU(n)}{\Sigma_g}$ for any $n \ge 1$. First, denote by
\[ c \colon \SU(n)^2 \to \SU(n), \quad \overline{c} \colon \PU(n)^2 \to \SU(n), \quad \tilde{c} \colon \U(n)^2 \to \SU(n) \]
the various commutator maps.

\begin{proposition}
    \label{prop:Un_in_terms_of_PUn}
    For any $g \ge 0$ and $n \ge 1$, the rational cohomology of the $\U(n)$-representation variety of $\Sigma_g$ is given by
    \[ H^*(\Rep{\U(n)}{\Sigma_g}; \QQ) = H^*(\Rep{\PU(n)}{\Sigma_g}^0; \QQ) \otimes (\QQ[0] \oplus \QQ[-1])^{\otimes 2g} , \]
    where $\Rep{\PU(n)}{\Sigma_g}^0$ denotes the identity component of $\Rep{\PU(n)}{\Sigma_g}$.
\end{proposition}
\begin{proof}
    Denote by $\gamma, \overline{\gamma}, \tilde{\gamma}$ the maps $c^g, \overline{c}^g, \tilde{c}^g$ composed with the multiplication $\SU(n)^g \to \SU(n)$, respectively. Consider the following commutative diagram in which all squares are cartesian.
    \[ \begin{tikzcd}[column sep=-1.0em, row sep=1.0em]
        & \gamma^{-1}(1) \times \U(1)^{2g} \arrow{rr} \arrow{ld} \arrow{dd} &&  (\SU(n) \times \U(1))^{2g} \arrow{ld} \arrow{dd} \\
        \tilde{\gamma}^{-1}(1) \arrow{rr} \arrow{dd} && \U(n)^{2g} \arrow{dd} & \\
        & \gamma^{-1}(1) \arrow{rr} \arrow{ld} \arrow[bend left=20]{lddd} && \SU(n)^{2g} \arrow{ld} \arrow[bend left=20]{lddd}{\gamma}  \\
        \overline{\gamma}^{-1}(1) \arrow{rr} \arrow{dd} && \PU(n)^{2g} \arrow{dd}{\overline{\gamma}} & \\ \\
        \{ 1 \} \arrow{rr} && \SU(n)
    \end{tikzcd} \]
    In particular, $\tilde{\gamma}^{-1}(1)$ is the quotient of $\gamma^{-1}(1) \times \U(1)^{2g}$ by $(\ZZ/n\ZZ)^{2g}$. But, since $\ZZ/n\ZZ$ acts on $\U(1)$ by translation, it acts trivial on its cohomology. Therefore, the cohomology of $\tilde{\gamma}^{-1}(1)$ is simply the tensor product of the cohomology of $\gamma^{-1}(1) / (\ZZ/n\ZZ)^{2g} = \overline{\gamma}^{-1}(1) = \Rep{\PU(n)}{\Sigma_g}^0$ and the cohomology of $\U(1)^{2g}$, the latter of which is $(\QQ[0] \oplus \QQ[-1])^{\otimes 2g}$.
\end{proof}

\begin{remark}
    Note that a relation as in \cref{prop:Un_in_terms_of_PUn} need not hold for the $\U(n)$- and $\PU(n)$-representation varieties of any finitely generated group $\Gamma$. For instance, for $\Gamma = \ZZ/2\ZZ$, we have $|\Rep{\U(n)}{\Gamma}| = 2^n$ and $|\Rep{\PU(n)}{\Gamma}^0| = 2^{n - 1}$ for all $n \ge 1$.
\end{remark}




\section{\texorpdfstring{$\SU(2)$}{SU(2)}-character stacks}

In this section, we turn our attention to the $\SU(2)$-character stacks $\CharStck{\SU(2)}{\Sigma_g}$ of the closed orientable surfaces $\Sigma_g$ for various genera $g$. Even though the computations in \cref{sec:SU2_representation_varieties} only deal with the $\SU(2)$-\textit{representation variety} of $\Sigma_g$ (that is, not taking into account the action of $G$ by conjugation), it turns out that these computations can still largely be used to determine the cohomology of $\CharStck{\SU(2)}{\Sigma_g}$ and the twisted $\SU(2)$-character stacks $\CharStckTw{G}{\Sigma_g}{C}$.

Let $G = \SU(2)$ and write $G/G$ for the quotient stack of the action of $G$ acting on itself by conjugation. Denote by $i \colon \{ 1 \} / G \to G / G$ and $j \colon U / G \to G / G$ the closed immersion induced by the unit in $G$, and its open complement, respectively. Furthermore, write $q \colon * \to BG$, $q_U \colon U \to U/G$ and $q_G \colon G \to G/G$ for the quotient maps, and write $\pi \colon BG \to *$ for the final morphism to the point.

\begin{proposition}
    \label{prop:stacky_commutator_map}
    \[ Z_G \left(\bdgenus \circ \bdunit \right)(\uQQ) = i_* \uQQ[-3]^{\oplus 2} \oplus \Cone \Big( i_* \uQQ[-1] \oplus i_* \uQQ[-3] \xrightarrow{(\delta \; \varepsilon)} j_! (q_U)_* \uQQ \Big) \]
\end{proposition}
\begin{proof}
    Analogous to \cref{subsection:commutator_map}, write $c \colon G^2 / G \to G/G$ for the map induced by the commutator map $(A, B) \mapsto [A, B]$. Let us understand $\mathcal{F} \defequality Z_G \left( \bdgenus \circ \bdunit \right) (\uQQ) = c_! \uQQ$ through the distinguished triangle
    \[ j_! j^* \mathcal{F} \to \mathcal{F} \to i_* i^* \mathcal{F} \xrightarrow{+} . \]
    As in the proof of \cref{prop:commutator_at_identity}, we have that $i^* \mathcal{F} = c'_! (i')^* \uQQ = c'_! \uQQ = c'_* \uQQ$, where $i'$ and $c'$ are given by the cartesian square
    \[ \begin{tikzcd}
        X/G \arrow{r}{i'} \arrow{d}{c'} & G^2/G \arrow{d}{c} \\ \{ 1 \} / G \arrow{r}{i} & G/G
    \end{tikzcd} \]
    with $X \defequality \{ 1 \} \times_G G^2 = \big\{ (A, B) \in G^2 \mid AB = BA \big\}$. Note that the subset $F \subset X$ of fixed points under the action of the maximal torus $T \subset G$ is equal to $F = \U(1)^2$. Hence, $\dim H^*(F; \QQ) = 4 = \dim H^*(X; \QQ)$, where the last equality follows by \cref{prop:commutator_at_identity}. This shows that the action of $T$ on $X$ is equivariantly formal \cite[Lemma C.24]{GuilleminGinzburgKarshon2002} and also the action of $G$ on $X$ is equivariantly formal \cite[Proposition C.26]{GuilleminGinzburgKarshon2002}, and hence $c'_* \QQ = \uQQ \boxtimes H^*(X; \QQ) = \uQQ[0] \oplus \uQQ[-2] \oplus \uQQ[-3]^{\oplus 2}$.

    Regarding $j^* \mathcal{F}$, one follows the proof of \cref{prop:commutator_at_general} to see that $j^* \mathcal{F} = c''_! \uQQ = c''_* \uQQ$ where $c''$ is given by the cartesian diagram
    \[ \begin{tikzcd}
        (U \times_G G^2)/G \arrow{r}{j'} \arrow{d}{c''} & G^2/G \arrow{d}{c} \\ U / G \arrow{r}{j} & G/G
    \end{tikzcd} \]
    Recall that the projection $U \times_G G^2 \to U$ is a trivial fiber bundle with fiber $\SO(3)$, and that the action of $G$ on the fiber $\SO(3)$ is given by left multiplication. Therefore, we have a commutative diagram
    \[ \begin{tikzcd}[row sep=0.5em]
        U \arrow[bend right=15]{rd} \arrow{rr}{q_U} & & U / G \\ & U / \{ \pm 1 \} = (U \times_G G^2) / G \arrow[swap, bend right=15]{ru}{c''} &
    \end{tikzcd} \]
    where $\{ \pm 1 \}$ acts trivially on $U$, which implies that $j^* \mathcal{F} = c''_* \uQQ = (q_U)_* \uQQ$.

    Finally, one follows the proof of \cref{prop:pushforward_along_commutator} to find that the connecting morphism $i_* i^* \mathcal{F}[-1] \to j_! j^* \mathcal{F}$ is given by
    \[ i_* \uQQ[-1] \oplus i_* \uQQ[-3] \oplus i_* \uQQ[-4]^{\oplus 2} \xrightarrow{\smatrix{\delta & \varepsilon & 0 & 0}} j_! (q_U)_* \uQQ . \qedhere \]
\end{proof}

For convenience, denote the $\Cone \Big( i_* \uQQ[-1] \oplus i_* \uQQ[-3] \xrightarrow{(\delta \; \varepsilon)} j_! (q_U)_* \uQQ \Big)$ by $\mathcal{V}$. 

\begin{definition}
    Define a sequence of objects $\mathcal{V}_n \in D(G/G)$ as follows. Put $\mathcal{V}_0 = i_* \uQQ$ and inductively define $\mathcal{V}_n = Z_G \left(\bdmultiplication[0.5]\right)(\mathcal{V}, \mathcal{V}_{n - 1})$ for all $n \ge 1$. In particular, $\mathcal{V}_1 = \mathcal{V}$.
\end{definition}

Now, in order to understand the cohomology of the $\SU(2)$-character stacks $\CharStck{G}{\Sigma_g}$, we must understand the objects $i^* \mathcal{V}_n$ in $D(BG)$. To give a good description, we make the following definition.

\begin{definition}
    \label{def:mathcal_G}
    Define a sequence of objects $\mathcal{G}_n \in D(BG)$ as follows. Put $\mathcal{G}_1 = q_* \uQQ$ and inductively define $\mathcal{G}_n = \Cone(\mathcal{G}_{n - 1}[-1] \xrightarrow{\alpha_n} q_* \uQQ[-4(n - 1)])$ for all $n \ge 2$, where $\alpha_n$ is a non-zero morphism. Note that this is well-defined, as the following proposition shows that $\Hom(\mathcal{G}_{n - 1}[-1], q_* \uQQ[-4(n - 1)]) = \Hom(q^* \mathcal{G}_{n - 1}, \uQQ[-4(n - 1) + 1]) = \QQ$ for all $n \ge 2$.
\end{definition}

\begin{proposition}
    \label{prop:properties_mathcal_G}
    For every $n \ge 1$, we have
    \begin{enumerate}[label=(\roman*)]
        \item $q^* \mathcal{G}_n = \uQQ \oplus \uQQ[-4n + 1]$
        \item $\pi_* \mathcal{G}_n = \pi_! \mathcal{G}_n = \bigoplus_{k = 0}^{n - 1} \QQ[-4k]$
    \end{enumerate}
\end{proposition}
\begin{proof}
    We prove these statements by induction on $n$. For $n = 1$, they follow from the definition $\mathcal{G}_1 = q_* \uQQ$. Let us now prove the statements for $n \ge 2$ assuming the statements hold for $n - 1$. Applying $q^*$ to the defining distinguished triangle for $\mathcal{G}_n$, we obtain
    \[ \begin{tikzcd}[row sep=0.5em]
        q^* \mathcal{G}_{n - 1}[-1] \arrow{r} & q^* q_* \uQQ[-4(n - 1)] \arrow{r} & q^* \mathcal{G}_n \arrow{r}{+} & \hphantom{} \\
        \uQQ[-1] \oplus \uQQ[-4(n - 1)] \arrow[equals]{u} \arrow{r}{{\smatrix{0 \; 1 \\ 0 \; 0}}} & \uQQ[-4(n - 1)] \oplus \uQQ[-4(n - 1) + 3] \arrow[equals]{u} & &
    \end{tikzcd} \]
    from which follows that $q^* \mathcal{G}_n = \uQQ \oplus \uQQ[-4n + 1]$, which proves \textit{(i)}. Similarly, applying $\pi_!$ to the defining distinguished triangle for $\mathcal{G}_n$, we obtain
    \[ \begin{tikzcd}[row sep=0.5em]
        \pi_! \mathcal{G}_{n - 1}[-1] \arrow{r} & \pi_! q_* \uQQ[-4(n - 1)] \arrow{r} & \pi_! \mathcal{G}_n \arrow{r}{+} & \hphantom{} \\
        \bigoplus_{k = 0}^{n - 2} \QQ[-4k - 1] \arrow{r}{0} \arrow[equals]{u} & \QQ[-4(n - 1)] \arrow[equals]{u} & &
    \end{tikzcd} \]
    from which follows that $\pi_! \mathcal{G}_n = \bigoplus_{k = 0}^{n - 2} \QQ[-4k]$, which proves \textit{(ii)}.
\end{proof}

The following two lemmas are needed for the proposition that succeeds them, in which we give a description of the pullbacks $i^* \mathcal{V}_n$ for $n = 1, 2, 3$.

\begin{lemma}
    \label{lemma:cone_Q_pi_Q}
    One has $\Hom_{\Sh(BG)}(\uQQ, q_* \uQQ) = \QQ$, and the non-trivial cone $\Cone(\uQQ \to q_* \uQQ)$ is isomorphic to $\uQQ[-3]$.
\end{lemma}
\begin{proof}
    The adjunction $q^* \dashv q_*$ yields $\Hom_{\Sh(BG)}(\uQQ, q_* \uQQ) = \Hom_{\Sh(*)}(\uQQ, \uQQ) = H^0(*) = \QQ$. Denoting the described cone by $\mathcal{F}$, we find that, for any point $x \colon * \to BG$, we have $x^* \mathcal{F} = \Cone(x^* \uQQ \to x^* (q_* \uQQ)) = \QQ[-3]$, because $x^* (q_* \uQQ) = H^*(\{ x \} \times_{BG} *; \QQ) = H^*(G; \QQ) = \QQ \oplus \QQ[-3]$.
    Hence, we must have $\mathcal{F} = \mathcal{L}[-3]$ for some local system $\mathcal{L}$ over $BG$ of rank one. Since $\pi_1(BG, *) = \pi_0(G)$ is trivial, we must have $\mathcal{L} = \uQQ$.
\end{proof}

\begin{lemma}
    \label{lemma:multiplication_that_kills_the_G_action}
    For any objects $\mathcal{F} \in D(G)$ and $\mathcal{G} \in D(G/G)$, we have
    \[ Z_G \left( \bdmultiplication[0.5] \right) \left( (q_G)_! \mathcal{F} \boxtimes \mathcal{G} \right) = (q_G)_! \mu(\mathcal{F}, q_G^* \mathcal{G}) . \]
\end{lemma}
\begin{proof}
    Consider the following commutative diagram, in which both squares are cartesian.
    \[ \begin{tikzcd}[column sep=6em]
        G \times G/G \arrow[swap]{d}{q_G \times \id_{G/G}} & G^2 \arrow[swap]{l}{(\pi_1, q_G \pi_2)} \arrow{d}{q_{G^2}} \arrow{r}{m} & G \arrow{d}{q_G} \\
        G/G \times G/G & G^2/G \arrow[swap]{l}{(\pi_1/G, \pi_2/G)} \arrow{r}{m/G} & G/G
    \end{tikzcd} \]
    The bottom row corresponds to the field theory of the bordism $\bdmultiplication[0.5]$, see \cref{prop:field_theory_surfaces}, so it follows that
    \begin{align*}
        Z_G \left( \bdmultiplication[0.5] \right) \left( (q_G)_! \mathcal{F} \boxtimes \mathcal{G} \right)
            &= (m/G)_! (\pi_1/G, \pi_2/G)^* ((q_G)_! \mathcal{F} \boxtimes \mathcal{G}) \\
            &= (m/G)_! (\pi_{G^2})_!(\pi_1, q_G \pi_2)^* (\mathcal{F} \boxtimes \mathcal{G}) \\
            &= (q_G)_! m_! (\pi_1, q_G \pi_2)^* (\mathcal{F} \boxtimes \mathcal{G}) \\
            &= (q_G)_! m_! (\pi_1^* \mathcal{F} \otimes \pi_2^* q_G^* \mathcal{G}) \\
            &= (q_G)_! \mu(\mathcal{F}, q_G^* \mathcal{G}). \qedhere
    \end{align*}
\end{proof}

\begin{proposition}
    \label{prop:properties_mathcal_V}
    The following equalities hold: 
    \begin{enumerate}[label=(\roman*)]
        \item $i^* \mathcal{V}_1 = \uQQ \oplus \uQQ[-2]$
        \item $i^* \mathcal{V}_2 = \uQQ \oplus \uQQ[-2] \oplus \uQQ[-4] \oplus \uQQ[-6] \oplus \mathcal{G}_1[-6]$
        \item $i^* \mathcal{V}_3 = \uQQ \oplus \uQQ[-2] \oplus \uQQ[-4] \oplus \uQQ[-6]^{\oplus 3} \oplus \uQQ[-8]^{\oplus 2} \oplus \mathcal{G}_2[-6] \oplus \mathcal{G}_2[-8]$
    \end{enumerate}
\end{proposition}
\begin{proof}
    By definition of $\mathcal{V}_1 = \mathcal{V}$, it is clear that $i^* \mathcal{V} = i^* i_* \uQQ \oplus i^* i_* \uQQ[-2] = \uQQ \oplus \uQQ[-2]$, which proves \textit{(i)}. For \textit{(ii)}, applying $i^* Z_G \left( \bdmultiplication[0.5] \right)(-, \mathcal{V}_1)$ to the distinguished triangle
    \begin{equation} \label{eq:triangle_for_mathcal_V} i_* \uQQ[-1] \oplus i_* \uQQ[-3] \xrightarrow{(\delta \; \varepsilon)} j_! (q_U)_* \uQQ \to \mathcal{V} \xrightarrow{+} \end{equation}
    we obtain the distinguished triangle
    \[ \begin{tikzcd}[row sep=0.5em]
        i^* \mathcal{V}_1[-1] \oplus i^* \mathcal{V}_1[-3] \arrow{r} & i^* Z_G \left( \bdmultiplication[0.5] \right) (j_! (q_U)_* \uQQ, \mathcal{V}_1) \arrow{r} & i^* \mathcal{V}_2 \arrow{r}{+} & \hphantom{} \\
        \uQQ[-1] \oplus \uQQ[-3]^{\oplus 2} \oplus \uQQ[-5] \arrow[equals]{u} \arrow{r}{\varphi} & q_* \uQQ[-3] \oplus q_* \uQQ[-6] \arrow[equals]{u} &
    \end{tikzcd} \]
    where the second equality is shown as follows. By \cref{lemma:multiplication_that_kills_the_G_action}, we have
    \begin{equation}
        \label{eq:going_to_non_equivariant_computations}
        i^* Z_G \left( \bdmultiplication[0.5] \right) (j_! (q_U)_! \uQQ, \mathcal{V}_1) = i^* (q_G)_! \mu((j')_! \uQQ, q_G^* \mathcal{V}_1) = q_! (i')^* \mu((j')_! \uQQ, q_G^* \mathcal{V}_1)
    \end{equation}
    where $i' \colon \{ 1 \} \to G$ and $j' \colon U \to G$ denote the closed immersion of the unit in $G$ and its open complement, respectively. But since $(j')_! \uQQ = \mathcal{U}$ and $q_G^* \mathcal{V}_1 = \mathcal{E} \oplus \mathcal{U}[-3]$ are known, the right-hand side is computed using \cref{prop:properties_mathcal_U,prop:properties_mathcal_E,lemma:mu_mathcal_E_U}. Finally, $q_! = q_*$ as $q$ is proper.
    
    Now, from the non-equivariant computations of \cref{sec:SU2_representation_varieties}, in particular \cref{prop:properties_mathcal_U,prop:properties_mathcal_E}, one finds that $q^* i^* \mathcal{V}_2 = (i')^* q_G^* \mathcal{V}_2 = (i')^* (\mathcal{E}_2 \oplus \mathcal{U}_2[-6] \oplus (i')_* \QQ[-6]^{\oplus 2}) = \QQ \oplus \QQ[-2] \oplus \QQ[-4] \oplus \QQ[-6]^{\oplus 2} \oplus \QQ[-9]$. Hence, the pullback of $\varphi$ along $q$ must be given by
    \[ \QQ[-1] \oplus \QQ[-3]^{\oplus 2} \oplus \QQ[-5] \xrightarrow{\smatrix{0 \; 1 \; 0 \; 0 \\ 0 \; 0 \; 0 \; 0 \\ 0 \; 0 \; 0 \; 0 \\ 0 \; 0 \; 0 \; 0}} \QQ[-3] \oplus \QQ[-6]^{\oplus 2} \oplus \QQ[-9] \]
    and subsequently we must have $\varphi = \smatrix{0 & 1 & 0 & 0 \\ 0 & 0 & 0 & 0}$, where the `$1$' indicates the non-trivial morphism in $\Hom(\uQQ[-3], q_* \uQQ[-3]) = \Hom(q^* \uQQ, \QQ) = \Hom(\QQ, \QQ) = \QQ$. Using \cref{lemma:cone_Q_pi_Q}, this proves \textit{(ii)}.
    
    One proves \textit{(iii)} in a similar way: applying $i^* Z_G \left( \bdmultiplication[0.5] \right)(-, \mathcal{V}_2)$ to \eqref{eq:triangle_for_mathcal_V} we obtain the distinguished triangle
    \[ \begin{tikzcd}[row sep=0.5em]
        i^* \mathcal{V}_2[-1] \oplus i^* \mathcal{V}_2[-3] \arrow{r} & i^* Z_G \left( \bdmultiplication[0.5] \right) (j_! (q_U)_* \uQQ, \mathcal{V}_2) \arrow{r} & i^* \mathcal{V}_3 \arrow{r}{+} & \hphantom{} \\
        \begin{tabular}{c} $\uQQ \oplus \uQQ[-3]^{\oplus 2} \oplus \uQQ[-5]^{\oplus 2} \oplus \uQQ[-7]^{\oplus 2}$ \\ $\oplus \, \uQQ[-9] \oplus \mathcal{G}_1[-7] \oplus \mathcal{G}_1[-9]$ \end{tabular} \arrow[equals]{u} \arrow{r}{\psi} & \begin{tabular}{c} $q_* \uQQ[-3] \oplus q_* \uQQ[-5]$ \\ $\oplus \, q_* \uQQ[-10] \oplus q_* \uQQ[-12]$ \end{tabular} \arrow[equals]{u} &
    \end{tikzcd} \]
    Again, from the non-equivariant computations of \cref{sec:SU2_representation_varieties}, one finds that $q^* i^* \mathcal{V}_3 = \QQ \oplus \QQ[-2] \oplus \QQ[-4] \oplus \QQ[-6]^{\oplus 4} \oplus \QQ[-8]^{\oplus 3} \oplus \QQ[-13] \oplus \QQ[-15]$. From this, one determines $q^* \psi$ and concludes that $\psi$ must be given by
    \[ \psi = \begin{pmatrix}
        0 & 1 & 0 & 0 & 0 & 0 & 0 & 0 & 0 & 0 \\
        0 & 0 & 0 & 1 & 0 & 0 & 0 & 0 & 0 & 0 \\
        0 & 0 & 0 & 0 & 0 & 0 & 0 & 0 & \alpha_2[-6] & 0 \\
        0 & 0 & 0 & 0 & 0 & 0 & 0 & 0 & 0 & \alpha_2[-8]
    \end{pmatrix} \]
    where $\alpha_2 \ne 0$ is the morphism as in \cref{def:mathcal_G}. This proves \textit{(iii)}.
\end{proof}

\begin{theorem}
    \label{thm:cohomology_SU2_character_stacks}
    The cohomology of the $\SU(2)$-character stacks of $\Sigma_g$ for $g = 1, 2, 3$ is given by
    \begin{enumerate}[label=(\roman*)]
        \item $H_c^*(\CharStck{\SU(2)}{\Sigma_1}; \QQ) = \Big( \QQ[-3]^{\oplus 2} \oplus \QQ[-2] \oplus \QQ \Big) \otimes \bigoplus_{k \ge 0} \QQ[4k + 3]$
        \item $H_c^*(\CharStck{\SU(2)}{\Sigma_2}; \QQ) = \QQ[-6] \oplus \Big( \QQ \oplus \QQ[-2] \oplus \QQ[-3]^{\oplus 4} \oplus \QQ[-4] \oplus \QQ[-5]^{\oplus 4} \oplus \QQ[-6]^{\oplus 5} \Big) \otimes \bigoplus_{k \ge 0} \QQ[4k + 3]$
        \item $H_c^*(\CharStck{\SU(2)}{\Sigma_3}; \QQ) = \QQ[-6] \oplus \QQ[-8] \oplus \QQ[-9]^{\oplus 6} \oplus \QQ[-10] \oplus \QQ[-12] \oplus \Big( \QQ \oplus \QQ[-2] \oplus \QQ[-3]^{\oplus 6} \oplus \QQ[-4] \oplus \QQ[-5]^{\oplus 6} \oplus \QQ[-6]^{\oplus 15} \oplus \QQ[-7]^{\oplus 6} \oplus \QQ[-6]^{\oplus 14} \oplus \QQ[-9]^{\oplus 14} \Big) \otimes \bigoplus_{k \ge 0} \QQ[4k + 3]$
        \item The cohomology $H^*(\CharStck{\SU(2)}{\Sigma_g}; \QQ)$ for $g = 1, 2, 3$ is given by the above expressions replacing the terms $\QQ[4k + 3]$ by $\QQ[-4k]$.
    \end{enumerate}
\end{theorem}
\begin{proof}
    The cohomology with compact support of the $\SU(2)$-character stack of $\Sigma_g$ is given by $Z_G(\Sigma_g)(\uQQ)$. Hence, \textit{(i)}, \textit{(ii)} and \textit{(iii)} can be computed using \eqref{eq:ZG_Sigma_g_from_convolution_mathcal_F}, in which $\mathcal{F}$ is given by \cref{prop:stacky_commutator_map}, and $Z_G \left( \bdcounit \right) = \pi_! i^*$ applied to $\mathcal{V}_1, \mathcal{V}_2$ and $\mathcal{V}_3$ can be computed using \cref{prop:properties_mathcal_V} and the facts that $\pi_! \uQQ = \bigoplus_{k \ge 0} \QQ[4k + 3]$ (see \cref{ex:group_homology_and_cohomology_SU2}) and $\pi_! \mathcal{G}_n = \bigoplus_{k = 0}^{n - 1} \QQ[-4k]$ (see \cref{prop:properties_mathcal_G}).
    For \textit{(iv)}, note that the map $\CharStck{G}{\Sigma_g} \to *$ factors as $\CharStck{G}{\Sigma_g} \xrightarrow{f} BG \xrightarrow{\pi} *$. Since $f_* = f_!$, we find that $H^*(\CharStck{G}{\Sigma_g}; \QQ)$ can be computed as $\pi_* f_* \uQQ = \pi_* f_! \uQQ$, and note that $\pi_* \uQQ = \bigoplus_{k \ge 0} \QQ[-4k]$ (see \cref{ex:group_homology_and_cohomology_SU2}).
\end{proof}

\subsection{Twisted \texorpdfstring{$\SU(2)$}{SU(2)}-character stacks}

We conclude by computing the cohomology of the twisted $\SU(2)$-character stacks for $g = 1, 2, 3$, using the same strategy as in the previous section. Denote by $i_{-1} \colon \{ -1 \}/G \to G/G$ the closed immersion induced by the inclusion of $-1 \in G$.

\begin{proposition}
    The following equalities hold:
    \begin{enumerate}[label=(\roman*)]
        \item $i_{-1}^* \mathcal{V}_1 = \mathcal{G}_1$
        \item $i_{-1}^* \mathcal{V}_2 = \mathcal{G}_2 \oplus \mathcal{G}_2[-2]$
        \item $i_{-1}^* \mathcal{V}_3 = \mathcal{G}_1[-6] \oplus \mathcal{G}_3 \oplus \mathcal{G}_3[-2] \oplus \mathcal{G}_3[-4]$
    \end{enumerate}
\end{proposition}
\begin{proof}
    By definition of $\mathcal{V}_1 = \mathcal{V}$, it is clear that $i_{-1}^* \mathcal{V} = i_{-1}^* j_! (q_U)_* \uQQ = q_* \uQQ$, which proves \textit{(i)}. For \textit{(ii)}, applying $i_{-1}^* Z_G \left( \bdmultiplication[0.5] \right)(-, \mathcal{V}_1)$ to the distinguished triangle \eqref{eq:triangle_for_mathcal_V} we obtain the distinguished triangle
    \[ \begin{tikzcd}[row sep=0.5em]
        i_{-1}^* \mathcal{V}_1[-1] \oplus i_{-1}^* \mathcal{V}_1[-3] \arrow{r} & i_{-1}^* Z_G \left( \bdmultiplication[0.5] \right) (j_! (q_U)_* \uQQ, \mathcal{V}_1) \arrow{r} & i_{-1}^* \mathcal{V}_2 \arrow{r}{+} & \hphantom{} \\
        \mathcal{G}_1[-1] \oplus \mathcal{G}_1[-3] \arrow[equals]{u} \arrow{r}{\varphi} & q_* \uQQ[-4] \oplus q_* \uQQ[-6] \arrow[equals]{u} &
    \end{tikzcd} \]
    where the second equality is shown analogous to \eqref{eq:going_to_non_equivariant_computations}, only replacing $i'$ by $i'_{-1}$.
    From the non-equivariant computations of \cref{sec:SU2_representation_varieties}, one finds that $q^* i_{-1}^* \mathcal{V}_2 = (i'_{-1})^* q_G^* \mathcal{V}_2 = (i'_{-1})^* (\mathcal{E}_2 \oplus \mathcal{U}_2[-6] \oplus (i')_* \QQ[-6]^{\oplus 2}) = \QQ \oplus \QQ[-2] \oplus \QQ[-7] \oplus \QQ[-9]$.
    Hence, the pullback of $\varphi$ along $q$ must be given by
    \[ \QQ[-1] \oplus \QQ[-3] \oplus \QQ[-4] \oplus \QQ[-6] \xrightarrow{\smatrix{0 \; 0 \; 1 \; 0 \\ 0 \; 0 \; 0 \; 1 \\ 0 \; 0 \; 0 \; 0 \\ 0 \; 0 \; 0 \; 0}} \QQ[-4] \oplus \QQ[-6] \oplus \QQ[-7] \oplus \QQ[-9] \]
    and subsequently, we must have $\varphi = \smatrix{\alpha_2 & 0 \\ 0 & \alpha_2[-2]}$ with $\alpha_2$ as in \cref{def:mathcal_G}, proving \textit{(ii)}.

    One proves \textit{(iii)} in a similar way: applying $i_{-1}^* Z_G \left( \bdmultiplication[0.5] \right) (-, \mathcal{V}_2)$ to \eqref{eq:triangle_for_mathcal_V} we obtain the distinguished triangle
    \[ \begin{tikzcd}[row sep=0.5em]
        i_{-1}^* \mathcal{V}_2[-1] \oplus i_{-1}^* \mathcal{V}_2[-3] \arrow{r} & i_{-1}^* Z_G \left( \bdmultiplication[0.5] \right) (j_! (q_U)_* \uQQ, \mathcal{V}_2) \arrow{r} & i_{-1}^* \mathcal{V}_3 \arrow{r}{+} & \hphantom{} \\ \mathcal{G}_2[-1] \oplus \mathcal{G}_2[-3]^{\oplus 2} \oplus \mathcal{G}_2[-5] \arrow[equals]{u} \arrow{r}{\psi} & \begin{tabular}{c} $q_* \QQ[-3] \oplus q_* \QQ[-6]^{\oplus 2} \oplus q_* \QQ[-8]$ \\ $\oplus \, q_* \QQ[-10] \oplus q_* \QQ[-12]$ \end{tabular} \arrow[equals]{u} &
    \end{tikzcd} \]
    Again, from the non-equivariant computations of \cref{sec:SU2_representation_varieties}, one finds that $q^* i_{-1}^* \mathcal{V}_3 = \QQ \oplus \QQ[-2] \oplus \QQ[-4] \oplus \QQ[-6]^{\oplus 3} \oplus \QQ[-9]^{\oplus 3} \oplus \QQ[-11] \oplus \QQ[-13] \oplus \QQ[-15]$. From this, one determines $q^* \psi$ and concludes that $\psi$ must be given by
    \[ \psi = \begin{pmatrix}
        0 & 1 & 0 & 0 \\
        0 & 0 & 0 & 0 \\
        0 & 0 & 0 & 0 \\
        \alpha_3 & 0 & 0 & 0 \\
        0 & 0 & \alpha_3[-2] & 0 \\
        0 & 0 & 0 & \alpha_3[-4]
    \end{pmatrix} \]
    which proves \textit{(iii)}.
\end{proof}

Analogous to \cref{thm:cohomology_SU2_character_stacks}, we obtain the following theorem.

\begin{theorem}
    \label{thm:cohomology_twisted_SU2_character_stacks}
    The cohomology and cohomology with compact support of the twisted $\SU(2)$-character stacks for $g = 1, 2, 3$ agree, and are given by
    \begin{enumerate}[label=(\roman*)]
        \item $H^*(\mathfrak{X}^\textup{tw}_{\SU(2))}(\Sigma_1); \QQ) = \QQ$
        \item $H^*(\mathfrak{X}^\textup{tw}_{\SU(2))}(\Sigma_2); \QQ) = \QQ \oplus \QQ[-2] \oplus \QQ[-3]^{\oplus 4} \oplus \QQ[-4] \oplus \QQ[-6]$ 
        \item $H^*(\mathfrak{X}^\textup{tw}_{\SU(2))}(\Sigma_3); \QQ) = \QQ \oplus \QQ[-2] \oplus \QQ[-3]^{\oplus 6} \oplus \QQ[-4]^{\oplus 2} \oplus \QQ[-5]^{\oplus 6} \oplus \QQ[-6]^{\oplus 16} \oplus \QQ[-7]^{\oplus 6} \oplus \QQ[-8]^{\oplus 2} \oplus \QQ[-9]^{\oplus 6} \oplus \QQ[-10] \oplus \QQ[-12]$ 
        \qed
    \end{enumerate}
\end{theorem}

\begin{remark}
    Unfortunately, we have not found a systematic way to practically compute the cohomology of the (twisted) $\SU(2)$-character stacks for $g \ge 4$.
\end{remark}

\printbibliography

\end{document}